 \documentclass[12pt]{amsart}
\usepackage{amssymb,latexsym, amsmath, amsxtra}
\usepackage{graphicx}
\newtheorem{theorem}{Theorem}
\newtheorem{lemma}{Lemma}
\newtheorem{corollary}{Corollary}
\newtheorem{proposition}{Proposition}
\newtheorem{conjecture}{Conjecture}
\begin{document}

\title{The sixth power moment of Dirichlet $L$-functions}
\author {J. B. Conrey, H.  Iwaniec,  \and K. Soundararajan}
\address{American Institute of Mathematics, 360 Portage Ave, Palo Alto, CA 94306, USA; 
Department of Mathematics, Bristol University, Bristol BS8 1SN, UK} 
\email{conrey@aimath.org}
\address{Department of Mathematics, Rutgers University, Piscataway, NJ 08903, USA}
\email{iwaniec@math.rutgers.edu}
\address{Department of Mathematics, Stanford University, Stanford, CA 94305, USA}
\email{ksound@math.stanford.edu}

\thanks{Research  supported in part by the American Institute of Mathematics and by
the NSF grants DMS-1101774, DMS-1101575, and DMS 1001068} \maketitle

\section{Introduction}

\def\sumflat{\sideset{}{^\flat}\sum}
 \def\sumstar{\sideset{}{^*}\sum}
\def\cbar{\overline{\chi}}
\def\Lam{\Lambda}
   
\noindent  The study of moments of the Riemann zeta-function and 
related $L$-functions has a long history tracing back to the 
classical work of Hardy and Littlewood.    Hardy and Littlewood established an 
asymptotic formula for the second moment $\int_0^T |\zeta(\tfrac 12+it)|^2 dt$, and 
Ingham established an asymptotic formula for the fourth moment $\int_0^{T} |\zeta(\tfrac 12+it)|^4 dt$.  
Finding asymptotic formulae for the sixth or higher moments of the zeta-function remains an 
outstanding open problem.  The story for moments 
of $L$-functions in families is similar, and in many cases a few small 
moments have been evaluated asymptotically.   In general it is difficult even to conjecture 
an asymptotic formula, and it is only in recent years that a satisfactory picture of the 
structure of moments has emerged.

The breakthrough came from the work of Keating and Snaith (\cite{KS1}, \cite{KS2}) who modeled moments 
of $L$-values by values of the characteristic polynomials of large random matrices drawn 
from appropriate classical compact groups.  The choice of the group was 
suggested by the work of Katz and Sarnak (\cite {KaSa}) on the symmetry types for the 
distribution of zeros of $L$-functions in families.    
In this way Keating and Snaith identified the leading order asymptotics for 
the $2k$-th moment of $\zeta(\tfrac 12+it)$, and their conjecture agreed with 
the results of Hardy and Littlewood, and Ingham for $k=1$ and $2$, and 
conjectures derived (using heuristics for shifted divisor problems) 
by Conrey and Ghosh (for $k=3$, \cite{CG}) and Conrey and Gonek (for $k=4$, \cite{CGo}).  
While the Keating-Snaith conjecture gives only the leading order 
term for the moments of $L$-functions, subsequent work by 
Conrey, Farmer, Keating, Rubinstein and Snaith (\cite {CFKRS}) has led to more precise 
conjecture for integer moments where the entire asymptotic 
expansion is identified.   An alternative approach, based on multiple Dirichlet series, 
and leading to the same conjectures was proposed by Diaconu, Goldfeld and Hoffstein 
\cite{DGH}.  
Towards these conjectures, in many situations we have a lower bound 
for moments of the conjectured order of magnitude (see \cite{RuS1}, \cite{RuS2}),  and assuming the 
truth of the Generalized Riemann Hypothesis we have a corresponding  
upper bound of almost the right order of magnitude (\cite{Sou1}).  Further, there is extensive 
numerical data  supporting the conjectures of  \cite{CFKRS}, and here 
it is important to know the full asymptotic expansion because in the range of computations 
asymptotically lower order terms still contribute significantly.  

 For the sixth moment of $\zeta(\tfrac 12+it)$, the Keating-Snaith 
 conjecture (in this case, the same conjecture was made earlier by 
 Conrey and Ghosh) predicts that 
\begin{eqnarray}\label{eqn:sixth}
\int_0^T |\zeta(1/2+it)|^6 ~dt \sim 42 \prod_p\left(1-\frac 1 p
\right)^4\left(1+\frac 4 {p}+\frac 1 {p^2}\right)
T\frac{\log^9T}{9!}.
\end{eqnarray}
The more precise version due to Conrey, Farmer, Keating,  
Rubinstein and Snaith predicts that for any
$\epsilon>0$,
\begin{eqnarray}\label{eqn:cfkrs}
\int_0^T |\zeta(1/2+it)|^6 ~dt = \int_0^{T} P_3\left(\log
\frac{t}{2\pi}\right)~dt +O(T^{1/2+\epsilon})\end{eqnarray}
 where $P_3$ is a polynomial
of degree $9$ whose exact coefficients are specified as complicated
infinite products and series over primes, and which is given approximately by 
\begin{eqnarray*}&&P_3(x)\approx  0.000005708 \,x^9 + 0.0004050 \,x^8 + 0.01107 \,x^7
+ 0.1484 \,x^6\\
&&\qquad + 1.0459 \,x^5 + 3.9843 \,x^4 + 8.6073 \,x^3 + 10.2743
\,x^2 +  6.5939 \,x + 0.9165.
\end{eqnarray*}
As remarked earlier, the precise conjecture \eqref{eqn:cfkrs} is better 
suited for numerical testing than \eqref{eqn:sixth}.  
 For example,
\begin{eqnarray*}  &&
\int_0^{2350000}|\zeta(1/2+it)|^6 ~dt   \approx 3317496016044.9\approx 3.3\times
10^{12}
\end{eqnarray*}
and this compares well with 
\begin{eqnarray*}&&
\int_0^{2350000} P_3\left(\log \frac{t}{2\pi}\right)~dt \approx
3317437762612.4,
\end{eqnarray*}
whereas
\begin{eqnarray*}&&
 42 \prod_p\left(1-\frac 1 p
\right)^4\left(1+\frac 4 {p}+\frac 1 {p^2}\right)\times
2350000\times \frac{(\log 2350000)^9}{9!} = 4.22\times 10^{11}
\end{eqnarray*}
is nowhere near the prediction. 

The proof of formula (\ref{eqn:sixth}) appears beyond the reach of
current technology.  Our goal in this paper is to establish an 
analog of (\ref{eqn:cfkrs}) for Dirichlet $L$-functions suitably averaged.  
Our formula agrees exactly with the conjecture
of \cite{CFKRS} and so provides, we hope, a new glimpse into the
mechanics of moments.  We first give a corollary of our work, postponing 
the more precise technical result to the next section.  

Let $\chi \pmod q$ be an even, primitive Dirichlet character and let
\begin{eqnarray*}L(s,\chi)=\sum_{n=1}^\infty \frac{\chi(n)}{n^s} 
= \prod_{p} \Big(1- \frac{\chi(p)}{p^s}\Big)^{-1} 
\end{eqnarray*}
be its associated $L$-function.  
This $L$-function satisfies the functional
equation
$$
\Lambda(\tfrac 12 +s,\chi):=\Big(\frac{q}{ \pi}\Big)^{s/2}\Gamma\Big(\frac 14+ \frac s2\Big)L(\tfrac 12 +s,\chi) =\epsilon_\chi
\Lambda(\tfrac 12-s,\overline{\chi}) 
$$
where $\epsilon_\chi$ is a complex number of absolute value 1.   Throughout 
the paper we shall use $\sumflat$ to indicate that a sum is over primitive, even 
Dirichlet characters, and $\phi^{\flat}(q)$ will denote the number of primitive 
even Dirichlet characters $\pmod q$.   The restriction to even characters is 
merely a matter of convenience, and we could equally well consider odd characters 
making appropriate changes to our argument (for example, the $\Gamma$-factor in the 
functional equation will be different).   In analogy  with \eqref{eqn:sixth}, 
we have the following conjecture from \cite{CFKRS} (note that when $q \equiv 2\pmod 4$ there 
are no primitive Dirichlet characters $\pmod q$):

\begin{conjecture}   Put
\begin{equation} 
\label{eqn:a3} 
a_3=\prod_{p} \Big(1-\frac 1
p\Big)^{4}\Big(1+\frac4 p+\frac 1{p^2}\Big).
\end{equation}
Then, as $q\to \infty$ with $q\not\equiv 2 \pmod 4$, 
$$
 \frac{1}{\phi^\flat(q)}\sumflat_{\chi \bmod q}| L(\tfrac 12,\chi)|^6
 \sim 42 a_3 \prod_{p\mid
 q}\frac{\Big(1-\frac{1}{p}\Big)^{5}}{
  \Big(1+\frac4 p+\frac 1{p^2}\Big)}\frac{(\log q)^9}{9!}.
$$
\end{conjecture}

Towards this Conjecture, we shall establish:  

\begin{corollary}  For large $Q$ we have 
\begin{align*}
 \sum_{ q\le Q }\ \ 
\sumflat_{ \chi\bmod q} & \int_{-\infty}^{\infty}
|\Lambda(\tfrac 12+iy,\chi)|^6~dy \\
&\sim 42 a_3  \sum_{ q\le Q}
\prod_{p\mid q}\frac{(1-\frac{1}{p})^{5}}{
  (1+\frac4 p+\frac 1{p^2})} \phi^\flat(q)\frac{(\log q)^9}{9!}
\int_{-\infty}^\infty \Big|\Gamma\Big(\frac{1/2+iy}{2}\Big)\Big|^6~dy.
\end{align*}
 \end{corollary}
 
 The number of primitive characters $\pmod q$ is given by 
 $\phi^*(q) = \sum_{dr=q} \mu(d)\phi(r)$, and the number of 
 even primitive characters is $\phi^{\flat}(q) = \phi^*(q)/2+ O(1)$.  
 Thus one can express the main term in Corollary 1 as 
$$
\sim 42 \widetilde {a_3}\frac{Q^2}{2} \frac{\log ^9 Q}{9!}  \int_{-\infty}^\infty \Big|\Gamma\Big(\frac{1/2+iy}2\Big)\Big|^6~dy
$$ 
where
$$
\widetilde{a_3}=\prod_p \Big(1-\frac{1}{p}\Big)^5\Big(1+\frac5 {p}-\frac{5}{p^2}
+\frac{14}{p^3}-\frac{15}{p^4}+\frac{5}{p^5}+\frac{4}{p^6} -
\frac{4}{p^7}+\frac{1}{p^8}\Big).
$$

However, the form in which we have written Corollary 1 is more suggestive as 
it reveals that we have established an average form of Conjecture 1, by 
introducing an average over the moduli $q$ and also an average over 
points $1/2+iy$.   The average over $q$ significantly increases the size of our 
family of $L$-functions -- from a family of about $q$ $L$-functions of 
conductor $q$ we move to a family of about $Q^2$ $L$-functions with conductor up to $Q$.  
The average over points $1/2+iy$ is more benign -- the completed $L$-function $\Lambda(1/2+iy)$ 
decays exponentially as $|y|$ increases, and so this average should be thought of as involving 
only points within a constant distance from the real axis.   Nevertheless the average 
over $y$ is necessary for our argument to work, and it would be nice to develop a 
corresponding result just at the point $1/2$.   In particular, because of the additional averaging over $y$, 
we do not conclude anything new about non-vanishing results at the central point.

Recall the large sieve inequality
\begin{eqnarray*}
\sum_{q\le Q}\frac{q}{\phi(q)} \sideset{}{^*}\sum_{\chi \bmod q}
\big|\sum_{n=1}^N a_n\chi(n)\big|^2 \le (Q^2+N)\sum_{n=1}^N |a_n|^2
\end{eqnarray*}
where $\sumstar$ indicates that the sum is restricted to  primitive
characters.  As an application of the large sieve, Huxley \cite{Hux} proved that
$$
\sum_{q\le Q} ~\sumstar_{\chi \bmod q}|L(\tfrac 12,\chi)|^6\ll Q^2 \log ^9
Q
$$
and
$$
\sum_{q\le Q} ~\sumstar_{\chi \bmod q}|L(\tfrac 12,\chi)|^8\ll Q^2 \log
^{16} Q.
$$
Our work may be seen as a refinement of Huxley's sixth moment estimate, and in 
this regard the current paper is a companion to \cite{CIS1} and  where we 
develop similar ``asymptotic large sieves" in other contexts.   A challenging problem is 
to obtain a similar asymptotic formula in place of Huxley's estimate for the 
eighth moment.

{\bf Acknowledgments.}  We thank Matthew Young for several useful suggestions.

\section{The sixth moment with shifts}

 \def\bsalpha{\boldsymbol{\alpha}} 
\def\bsbeta{\boldsymbol{\beta}} 

In this section we recall the conjecture from \cite{CFKRS} for the $2k$-th moment of 
Dirichlet $L$-functions, and then state our main theorem from which 
Corollary 1 follows.  Let ${{ \bsalpha}} = (\alpha_1,\ldots,\alpha_{k})$, and $\bsbeta= 
(\beta_1,\ldots,\beta_k)$ be two 
vectors of $k$ complex numbers each, and we suppose that $|{\text Re }(\alpha_j)|, |{\text Re }(\beta_j)| 
\le 1/4$ for all $j$.  We shall also write $\beta_j =\alpha_{k+j}$, and think of 
the pair $(\bsalpha, \bsbeta)$ as defining a $2k$-tuple $(\alpha_1,\ldots,\alpha_{2k})$.  
Given $\bsalpha$ and $\bsbeta$, we define 
$$ 
\Lambda(s, \chi; \bsalpha, \bsbeta) = \prod_{j=1}^{k} \Lambda(s+\alpha_j,\chi) \Lambda(s-\beta_j, 
\overline{\chi}),
$$ 
and set 
$$ 
\Lambda(\chi;\bsalpha, \bsbeta) =\Lambda(\tfrac 12, \chi; \bsalpha,\bsbeta).
$$ 

Note that any permutation of the coordinates of $\bsalpha$, and any permutation 
of the coordinates of $\beta$, leaves $\Lambda(s,\chi;\bsalpha,\bsbeta)$ unaltered.  
Moreover the functional equation gives $\Lambda(s,\chi;\bsalpha,\bsbeta)= 
\Lambda(1-s,\chi;\bsbeta,\bsalpha)$.  When $s=1/2$, the function $\Lambda(\chi;\bsalpha,\bsbeta)$ 
satisfies further symmetry properties.   To see this, note that the permutation group $S_{2k}$ 
acts naturally on the pair $(\bsalpha,\bsbeta)$.  Writing this pair as a $2k$-tuple, for $\pi \in S_{2k}$ 
we define  $\pi(\bsalpha,\bsbeta) = (\alpha_{\pi(1)},\ldots, \alpha_{\pi(2k)})$ and then 
take the first $k$ coordinates to be $\pi(\bsalpha)$ and the second $k$ coordinates  to 
be $\pi(\bsbeta)$.  Now note that the functional equation 
$$
\Lambda(\chi;\bsalpha, \bsbeta) = \Lambda(\chi;\pi (\bsalpha),\pi(\bsbeta)), 
$$ 
holds for all permutations $\pi \in S_{2k}$.   Thus $\Lambda(\chi;\bsalpha,\bsbeta)$ is 
invariant under $S_{2k}$ while $\Lambda(s,\chi;\bsalpha,\bsbeta)$ is 
invariant under the subgroup $S_k \times S_k$.

We wish to state the conjecture from \cite{CFKRS} on the average value of $\Lambda(\chi;
\bsalpha, \bsbeta)$.  Naturally, the conjectured answer must 
share the $S_{2k}$-symmetry described above.  

Define 
\begin{equation}
\label{delta}
\delta({\bsalpha},\bsbeta)=\frac 12 \sum_{j=1}^{k} (\alpha_j -\beta_j),
\end{equation}
and put 
\begin{equation} 
\label{Galpha} 
G(s;\bsalpha,\bsbeta) =  \prod_{j=1}^{k} \Gamma\Big( \frac s2 +\frac{\alpha_j}{2}\Big) 
\Gamma\Big( \frac s2- \frac{\beta_{j}}{2}\Big), 
\end{equation} 
so that
$$
\Lambda(\chi;\bsalpha,\bsbeta) =\Big (\frac q\pi\Big)^{\delta(\bsalpha,\bsbeta)} G(\tfrac 12;
\bsalpha,\bsbeta) \prod_{j=1}^{k} L(\tfrac 12 +\alpha_j, \chi) L(\tfrac 12-\beta_j,\overline{\chi}). 
$$

We define a generalized sum-of-divisors function by 
\begin{equation} 
\label{sigma} 
\sigma(n;\bsalpha) = \sum_{n=n_1 \cdots n_k} n_1^{-\alpha_1} \cdots n_{k}^{-\alpha_k}, 
\end{equation} 
so that we may write, when the real part of $s$ is sufficiently large 
$$ 
\prod_{j=1}^{k} L(s+\alpha_j, \chi) L(s-\beta_j,\overline{\chi}) = 
\sum_{m,n =1}^{\infty} \frac{\sigma(m;\bsalpha)}{m^s} \frac{\sigma(n;-\bsbeta)}{n^s} 
\chi(m) \overline{\chi(n)}. 
$$ 
If we average the above over even primitive characters $\chi$, a candidate for the 
answer would be the diagonal terms $m=n$ with $(m,q)=(n,q)=1$:   namely,
$$ 
\sum_{\substack{{n=1}\\ {(n,q)=1}}}^{\infty} \frac{\sigma(n;\bsalpha)\sigma(n;-\bsbeta)}{n^{2s}}=
\prod_{p\nmid q} {\mathcal B}_p(s;\bsalpha,\bsbeta), 
$$ 
where the Euler factor ${\mathcal B}_p$ is given by 
\begin{equation} \label{eqn:calB}
\mathcal B_p(s;\bsalpha,\bsbeta):=\sum_{r=0}^{\infty} \frac{\sigma(p^r;\bsalpha) \sigma(p^r;-\bsbeta)}{p^{2rs}} = 
 \int_0^1
\prod_{j=1}^k \Big(1- \frac{e(\theta)}{p^{s+\alpha_j}}\Big)^{-1} \prod_{\ell=1}^{k}
\Big( 1- \frac{e(-\theta)}{p^{s-\beta_\ell}} \Big)^{-1}~d\theta .
\end{equation}
The behavior of this Euler product can be understood by comparing it with 
an appropriate product of zeta functions.  
For a prime number $p$ let $\zeta_p(x)  = (1-p^{-x})^{-1}$, 
and define 
\begin{equation} 
\label{eqn:calZ}
{\mathcal Z}_p(s;\bsalpha,\bsbeta) = 
\prod_{j=1}^k \prod_{\ell =1}^{k} \zeta_p (2s+\alpha_j-\beta_\ell) 
\qquad \text{and} 
\qquad \mathcal Z(s;\bsalpha,\bsbeta):=
\prod_{j=1}^k \prod_{\ell=1}^{k} \zeta(2s+\alpha_j-\beta_\ell).  
\end{equation}
Further, let 
\begin{equation}\label{eqn:calA}
{\mathcal A}(s;\bsalpha,\bsbeta):=\prod_p { \mathcal B}_p(s;\bsalpha,\bsbeta)\mathcal Z_p(s;\bsalpha,\bsbeta)^{-1}
\end{equation}
With a little calculation, we may see that the conditions 
$|\text{Re }\alpha_j|, |\text{Re }\beta_j|\le 1/4$ ensure that the Euler 
product for $\mathcal A$ converges absolutely. 
Let $\mathcal B_q=\prod_{p\mid q}\mathcal B_p$, and 
define 
\begin{equation}
\label{eqn:QST}
 \mathcal Q(q;\bsalpha,\bsbeta)=
\Big(\frac{q}{\pi}\Big)^{\delta(\bsalpha,\bsbeta)} G(\tfrac 12;\bsalpha,\bsbeta) \frac{\mathcal A \mathcal
Z}{\mathcal B_{q}}(\tfrac 12;\bsalpha,\bsbeta).
\end{equation}
This is our candidate, based on the diagonal contribution alone, for the 
average of $\Lambda(\chi;\bsalpha,\bsbeta)$.  

Note that ${\mathcal Q}(q;\bsalpha,\bsbeta)$ is symmetric under $S_k \times S_k$, 
but not under the full group $S_{2k}$; so it cannot be the 
full answer for the average of $\Lambda(\chi;\bsalpha,\bsbeta)$.  We symmetrize this by summing over 
all $\binom{2k}{k}$ cosets of $S_{2k}/(S_k \times S_k)$.  Thus we define 
\begin{equation}\label{eqn:calQ} 
\widetilde{ \mathcal Q}(q;\bsalpha,\bsbeta) :=  \sum_{\pi \in S_{2k}/(S_k\times S_k)}
\mathcal Q(q;\pi(\bsalpha), \pi(\bsbeta)), 
\end{equation} 
 and the conjecture of Conrey, Farmer, Keating, Rubinstein and Snaith \cite{CFKRS} is that this object 
 is in fact close to the average value of $\Lam(\chi;\bsalpha,\bsbeta)$.

\begin{conjecture} Assuming that the ``shifts'' $\bsalpha$, $\bsbeta$  
satisfy $|\text{Re }\alpha_j|,  |\text{Re }
\beta_j| \le 1/4,$ and $\text{Im }\alpha_j ,\text{Im }\beta_j \ll q^{1-\epsilon}$, we
conjecture that
\begin{equation}  
\label{eqn:conj}
 \sumflat_{\chi \bmod q} \Lambda(\chi;\bsalpha,\bsbeta)
=\phi^{\flat}(q)\widetilde{\mathcal Q}(q;\bsalpha,\bsbeta)(1+O(q^{-1/2+\epsilon}))
\end{equation}
where $\sumflat$ denotes a sum over even primitive characters.
\end{conjecture}

Even though ${\mathcal Q}(q;\bsalpha,\bsbeta)$ has singularities 
(when $\alpha_j=\beta_\ell$) the symmetrized $\widetilde{\mathcal Q}(q;\bsalpha,\bsbeta)$ 
is in fact analytic in $\alpha_j$, $\beta_\ell$.   Thus one can let all 
the shifts tend to zero in Conjecture 2, and obtain a conjecture 
for the $2k$-th moment  of $L(\tfrac 12,\chi)$.  We refer to \cite{CFKRS}
for the details of this calculation, and note that when $k=3$ this 
is what leads to Conjecture 1 stated earlier.

Given $\bsalpha$ we define $\bsalpha +s$ to be the 
translated $k$-tuple $(\alpha_1 +s, \ldots, \alpha_k+s)$, and 
similarly $\bsbeta+s = (\beta_1+s,\ldots, \beta_k+s)$.  
Now we are ready to state our main theorem.  

\begin{theorem}  Let $Q$ be large, and let $\bsalpha$ and $\bsbeta$ be $3$-tuples 
with $\alpha_j$, $\beta_j \ll 1/\log Q$.  Let $\Psi$ be a smooth function compactly 
supported in $[1,2]$.  Then
\begin{align*}
\sum_{q}\Psi\left(\frac q Q\right) \int_{-\infty}^\infty 
 &\sumflat_{\chi }  \Lambda(\chi;\bsalpha+iy, \bsbeta+iy)~dy\\
&=\sum_{q}\Psi\left(\frac q Q\right) \int_{-\infty}^\infty 
 \phi^\flat(q)\widetilde{\mathcal Q}(q;\bsalpha+iy,\bsbeta+iy)~dy
+O(Q^{19/10+\epsilon}).
\end{align*}
\end{theorem}

We could equally well prove a theorem for odd primitive characters.
The answer would be similar with just the Gamma-factors changed
slightly to reflect the difference in the functional equation for
odd primitive Dirichlet $L$-functions. When the shifts are all $0$,
this difference disappears in the leading order main term.

In what follows, we focus on establishing Theorem 1.  From this, 
Corollary 1 follows by letting the shifts tend to zero, 
and by making $\Psi$ approximate the characteristic function of $[1,2]$ 
and replacing $Q$ by $Q/2$, $Q/4$, $\ldots$.    Since the calculation 
of letting the shifts tend to zero is entirely analogous to the 
derivation of Conjecture 1 from Conjecture 2 in \cite{CFKRS}, we omit 
the details.

\section{The ``approximate" functional equation}

\noindent We formulate a general approximate functional equation for 
shifted products of $L$-functions, which we shall later specialize to 
the case when $k=3$.   Let 
\begin{equation} 
\label{H(s)} 
H(s;\bsalpha, \bsbeta) = \prod_{j=1}^{k} \prod_{\ell=1}^{k} 
\Big( s^2 -\Big(\frac{\alpha_j -\beta_\ell}{2}\Big)^2 \Big)^3,
\end{equation}
and put for any $\xi >0$ 
\begin{equation} 
\label{eqn:W}
W(\xi;\bsalpha,\bsbeta) = \frac{1}{2\pi i} \int_{(1)} G(\tfrac 12+ s;\bsalpha,\bsbeta) H(s;\bsalpha,\bsbeta) \xi^{-s} 
\frac{ds}{s}. 
\end{equation} 
Put 
\begin{equation}
\label{Lam0} 
\Lam_0(\chi;\bsalpha, \bsbeta)  =\Big(\frac{q}{\pi}\Big)^{\delta(\bsalpha,\bsbeta)} 
 \sum_{m,n=1}^{\infty} 
\frac{\sigma(m;\bsalpha)}{\sqrt{m}} \frac{\sigma(n;-\bsbeta)}{\sqrt{n}} 
\chi(m) \cbar(n) W\Big(\frac{mn\pi^{k}}{q^k};\bsalpha,\bsbeta\Big). 
\end{equation} 

\begin{proposition}
\label{AFE} With notation as above, we have 
$$ 
H(0;\bsalpha,\bsbeta) \Lam(\chi;\bsalpha,\bsbeta) = \Lam_0(\chi;\bsalpha,\bsbeta) + \Lam_0(\chi;\bsbeta,\bsalpha).
$$
\end{proposition} 
\begin{proof}   We begin with 
$$ 
\frac{1}{2\pi i} \int_{(1)} \Lam(\tfrac 12+s,\chi;\bsalpha, \bsbeta) H(s;\bsalpha, \bsbeta) \frac{ds}{s}.  
$$ 
We move the line of integration to Re$(s)=-1$, encountering a pole at $s=0$ 
which leaves the residue $\Lam(\chi;\bsalpha,\bsbeta)H(0;\bsalpha,\bsbeta)$.  
For the remaining integral 
on the $-1$-line, we use the functional equation $\Lam(\tfrac 12+s, \chi;\bsalpha,\bsbeta) 
= \Lam(\tfrac 12- s, \chi; \bsbeta,\bsalpha)$, together with $H(s;\bsalpha,\bsbeta) 
= H(-s;\bsbeta,\bsalpha)$ to conclude that this integral equals (replacing $-s$ by $w$)
$$ 
-\frac{1}{2\pi i} \int_{(1)} \Lam(\tfrac 12 +w, \chi;\bsbeta,\bsalpha) H(w;\bsbeta,\bsalpha)\frac{dw}{w}. 
$$ 
We conclude that 
$$ 
H(0;\bsalpha,\bsbeta) \Lam(\chi;\bsalpha,\bsbeta) 
= \frac{1}{2\pi i} \int_{(1)} \Big( \Lam(\tfrac 12+s,\chi;\bsalpha,\bsbeta) 
H(s;\bsalpha,\bsbeta) + \Lam(\tfrac 12+s,\chi;\bsbeta,\bsalpha) H(s;\bsbeta,\bsalpha)\Big) 
\frac{ds}{s}.
$$
Expanding the $L$-functions into their Dirichlet series we obtain the Proposition. 
\end{proof}

For any positive real numbers $\xi$, $\eta$ and $\mu$ let us define 
\begin{equation}
\label{eqn:3}
V_{\bsalpha,\bsbeta}(\xi, \eta; \mu)= \Big( \frac{\mu}{\pi} \Big)^{\delta(\bsalpha,\bsbeta)} 
\int_{-\infty}^{\infty} \Big( \frac{\eta}{\xi}\Big)^{it} W\Big( \frac{\xi\eta \pi^k}{\mu^k}; \bsalpha+it, 
\bsbeta+it\Big) dt. 
\end{equation} 
Further, set 
\begin{equation} 
\label{eqn:2} 
\Lambda_1(\chi;\bsalpha, \bsbeta)  = \sum_{m, n=1}^{\infty} \frac{\sigma(m;\bsalpha)}{\sqrt{m}} 
\frac{\sigma(n;-\bsbeta)}{\sqrt{n}} \chi(m) \cbar(n) V_{\bsalpha,\bsbeta} (m, n;q). 
\end{equation} 
 
\begin{proposition} 
\label{AFE2} 
With notation as above, we have 
$$ 
H(0;\bsalpha,\bsbeta) \int_{-\infty}^{\infty} \Lambda(\chi;\bsalpha+it,\bsbeta+it) dt = 
\Lam_1(\chi;\bsalpha,\bsbeta) + \Lam_1(\chi;\bsbeta,\bsalpha). 
$$
\end{proposition} 
\begin{proof}  This follows readily from Proposition \ref{AFE} and the 
observation $\sigma(n;\bsalpha+it) = \sigma(n;\bsalpha)n^{-it}$.  
\end{proof} 

Our proof of Theorem 1, starts from the approximate functional 
equation given in Proposition \ref{AFE2} with $k=3$.   We shall analyze 
the terms $\Lam_1(\chi;\bsalpha,\bsbeta)$ and $\Lam_1(\chi;\bsbeta,\bsalpha)$ 
when averaged over all characters $\chi$ with conductor of size $Q$.  Notice 
that while the original function $\Lam(\chi;\bsalpha,\bsbeta)$ has $S_{2k}$ 
symmetry, the terms arising from Proposition \ref{AFE2} have only $S_k\times S_k$ 
symmetry.  Thus at the outset, the symmetry of the final answer has been 
lost, and this is one reason that the analysis of the main terms in our 
argument is complicated and we have to work hard to recover 
the symmetry in the end.  It would be interesting to develop an approach which 
maintains the symmetry throughout the argument, but we don't know 
how to do this.  

In our later work, it will be useful to have an understanding of the weights $W$ and 
$V$ defined above.  

\begin{lemma} 
\label{weights}  The weight $W(\xi;\bsalpha,\bsbeta)$ is a smooth function of $\xi \in (0,\infty)$.  
For large values of $\xi$, and any non-negative integer $\nu$, we have $W^{(\nu)}(\xi;\bsalpha,\bsbeta) \ll_{\nu} \exp(-c_0 \xi^{\frac 1k})$ for some absolute constant $c_0>0$.  
Further, we have 
$$ 
V_{\bsalpha,\bsbeta}(\xi, \eta;\mu) \ll \Big(\frac{\mu}{\pi} \Big)^{{\rm{Re }}  \delta(\bsalpha,\bsbeta)} 
\exp\Big(  - c_1 \Big(\frac{\max(\xi,\eta)^2}{\mu^k}\Big)^{\frac 1k} \Big). 
$$
\end{lemma} 
\begin{proof}  For any non-negative integer $\nu$ and any $c>0$ we have 
$$ 
\frac{d^{\nu}}{d\xi^{\nu}} W(\xi;\bsalpha,\bsbeta) = 
\frac{1}{2\pi i} \int_{(c)} G(\tfrac 12+s;\bsalpha, \bsbeta) H(s;\bsalpha,\bsbeta) 
\Big(\frac{d^{\nu}}{d\xi^{\nu}} \xi^{-s} \Big) \frac{ds}{s}. 
$$ 
It follows that $W$ is smooth in $\xi$.  Moreover, for $\xi$ large we choose $c=\xi^{1/k}$, and then a calculation using Stirling's formula reveals that the above is $\ll_{\nu}\exp(-c_0 \xi^{1/k})$ as 
stated. 

To prove the estimate for $V_{\bsalpha,\bsbeta}$, assume without loss of generality that 
$\eta \ge \xi$.  In the definition \eqref{eqn:3} we let $z= it$ denote a 
complex variable, and insert the definition of $W$.  Thus, using $H(s;\bsalpha+z,\bsbeta_z) 
= H(s;\bsalpha,\bsbeta)$, we find that 
$$ 
V_{\bsalpha,\bsbeta} (\xi,\eta;\mu) 
=\Big(\frac{\mu}{\pi}\Big)^{\delta(\bsalpha,\bsbeta)} \frac{1}{i} \int_{-i\infty}^{i\infty} 
\Big(\frac{\eta}{\xi}\Big)^{z} \frac 1{2\pi i} \int_{(c) } G(\tfrac 12+s;\bsalpha+z,\bsbeta+z) 
H(s;\bsalpha,\bsbeta) \Big(\frac{\xi\eta \pi^k}{\mu^k}\Big)^{-s} \frac{ds}{s} .
$$ 
Now we may take the line of integration for $z$ to be any line Re$(z) =d$, and 
the line of integration for $s$ to be any $c>0$, and so long as $|d|\le c$ 
no poles of the $\Gamma$-functions present in $G$ will be encountered.  
In the situation when $\eta\ge \xi$, we choose $c= (\eta^2/\mu^k)^{1/k}$ and 
$d=-c$, and then a calculation with Stirling's formula gives our desired bound. 

\end{proof}  

To close this section, we remark on why the integral over $y$ is needed in 
our Theorem.  Specializing to the case $k=3$, in the approximate functional equation of Proposition \ref{AFE} 
the variables $m$ and $n$ are restricted so that their product is about size $q^3$, 
and it is possible for one variable to be very large (about size $q^3$) while the 
other remains small.  Such a range is not amenable to the argument we develop.  
In contrast, the integrated approximate functional equation of Proposition \ref{AFE2} 
leads (thanks to Lemma \ref{weights}) to terms where $m$ and $n$ are 
both at most $q^{\frac 32+\epsilon}$.

% \begin{eqnarray}\label{eqn:calB2}
%\sum_{g=0}^\infty\frac{\sigma_{-A}(p^g)\sigma_{-B}(p^g)}{p^g}=\mathcal B_p(A,B)
%\end{eqnarray}
%so that,
%if $\Re \alpha, \Re \beta >0$, then
 %\begin{eqnarray}\label{eqn:siggen} \sum_{(n,q)=1}
%\frac{\sigma_{-A}(n)\sigma_{-B}(n)}{n}= \frac{\mathcal A\mathcal
%Z(A,B)}{\mathcal B_q  (A,B)}.
%\end{eqnarray}

  %We note that $V(\lambda,\eta)$ is small unless both $\lambda$ and $\eta$ are both small, and
 %$\lambda/\eta \asymp 1$.

 \section{ First steps towards Theorem 1}

\noindent Henceforth we assume that $k=3$ and that the shifts $\bsalpha$ and $\bsbeta$ 
satisfy Re $\alpha_j, \text{Re }\beta_j \ll 1/\log Q$.  Our aim is to evaluate asymptotically
\begin{equation}
\label{eqn:4} 
\mathcal{I}=\mathcal{I}(\Psi,Q;\bsalpha, \bsbeta):= H(0;\bsalpha,\bsbeta) \sum_{ q }
\sumflat_{ \chi\bmod q} \Psi\Big(\frac qQ\Big) 
\int_{-\infty}^{\infty} \Lambda(\chi;\bsalpha+it,\bsbeta+it)dt, 
\end{equation} 
where $\Psi$ is a fixed smooth function, compactly supported in $[1,2]$. 
Using Proposition \ref{AFE2}, we may write 
\begin{equation} 
\label{Delta} 
 {\mathcal I}(\Psi,Q;\bsalpha, \bsbeta)=\Delta(\Psi,Q;\bsalpha,\bsbeta)+ \Delta(\Psi,Q;\bsbeta,\bsalpha)
 \end{equation} 
 where
 \begin{equation}\label{eqn:5}  
 \Delta(\Psi,Q;\bsalpha,\bsbeta) =\sum_{ q } \sumflat_{ \chi\bmod q}
 \Psi\Big(\frac qQ\Big)\Lambda_1(\chi;\bsalpha,\bsbeta).
  \end{equation}
  
  \begin{lemma}
  \label{orth}  Let $\phi^*(q)$ denote the number of primitive characters $\pmod q$, and 
  let $\phi^{\flat}(q)$ denote the number of even primitive characters $\pmod q$.  If $m$ and 
  $n$ are integers with $(mn,q)=1$ then 
  $$ 
  \sumflat_{\chi\pmod q} \chi(m) \cbar(n) = \frac 12 \sum_{{\substack{{q=dr}\\{r| (m\pm n)} }} }\mu(d)\phi(r), 
  $$ 
  where we sum over both choices of sign.  In particular, $\phi^*(q) = \sum_{dr=q} \mu(d) \phi(r)$, 
  and $\phi^{\flat}(q) = \phi^*(q)/2+ O(1)$.  
  \end{lemma}
  \begin{proof}  From the orthogonality relations for characters and M{\" o}bius inversion 
  we easily see that 
  $$ 
  \sumstar_{\chi\pmod q} \chi(m) \overline{\chi(n)} = \sum_{\substack{{q=dr}\\ {r|m-n} }} \mu(d) \phi(r),
  $$ 
  where the sum is over all primitive characters $\pmod q$.  Since an even character 
  can be detected by $\tfrac 12 (1+ \chi(-1))$ we obtain the formula for the 
  sum over even primitive characters.  \end{proof}

From the definitions \eqref{eqn:3}, \eqref{eqn:2} and \eqref{eqn:5}, and Lemma 
\ref{orth} we obtain that  
 \begin{equation} 
 \label{eqn:2.3} 
 \Delta(\Psi,Q;\bsalpha,\bsbeta) 
 =\frac 12  \sum_{ m,n =1}^{\infty} \frac{\sigma(m;\bsalpha)}{\sqrt{m}} \frac{\sigma(n;-\bsbeta)}
 {\sqrt{n}}
 \sum_{\substack{{d, r}  \\{ (dr,mn)=1} \\ { r|m\pm n}}}
 \phi(r)\mu(d)\Psi\Big(\frac{dr}Q\Big) V_{\bsalpha,\bsbeta}(m,n;dr).
 \end{equation}

 There is a diagonal contribution to (\ref{eqn:2.3}) from the terms $m=n$, 
 which we call $\mathcal D(\Psi,Q;\bsalpha,\bsbeta)$.
 For the terms $m\neq n$, we introduce a parameter $D$ and divide
 the sum in (\ref{eqn:2.3}) into two ranges depending on whether $d>  D$ or $d\le D$.
Call the first set of terms $\mathcal S(\Psi,Q;\bsalpha,\bsbeta)$ (for small values of
$r$), and the second $\mathcal G(\Psi,Q;\bsalpha,\bsbeta)$ (for greater values of $r$).
Thus, we have the decomposition 
\begin{equation} \label{eqn:mainsum}
\Delta(\Psi,Q;\bsalpha,\bsbeta)=\mathcal D(\Psi,Q;\bsalpha,\bsbeta)+
\mathcal S(\Psi,Q;\bsalpha,\bsbeta)+\mathcal G(\Psi,Q;\bsalpha,\bsbeta).
\end{equation}

Of these three terms, the diagonal contribution is the easiest to evaluate, and 
we treat it first. 

\begin{lemma} 
\label{diagonal} 
With notations as above, 
$$ 
\mathcal {D}(\Psi,Q;\bsalpha,\bsbeta) =  H(0;\bsalpha,\bsbeta) 
\sum_{q} \Psi\Big(\frac qQ\Big) \phi^{\flat}(q) \int_{-\infty}^{\infty} 
{\mathcal Q}(q;\bsalpha+ it, \bsbeta+it) dt + O(Q^{\frac 54+\epsilon}).
$$
\end{lemma}
\begin{proof}  If $(n,q)=1$ then 
$$
\frac 12 \sum_{\substack{{dr=q}\\ {r | n\pm n} } } \mu(d) \phi(r) = \phi^\flat(q),
$$ 
and so  we see that
\begin{align}\label{eqn:diag1} 
\mathcal D(\Psi,Q;\bsalpha, \bsbeta) = \sum_q \phi^\flat(q) \Psi\Big(\frac qQ\Big) 
\sum_{(n,q)=1} \frac{\sigma(n;\bsalpha)\sigma(n;-\bsbeta)}{n} V_{\bsalpha,\bsbeta}(n,n;q). 
\end{align}
Recalling the definition of $V$ from (\ref{eqn:3}) 
we see that the sum over $n$ above equals
$$ 
\int_{-\infty}^{\infty} \frac{1}{2\pi i} 
\int_{(1)} G(\tfrac 12+s;\bsalpha+it,\bsbeta+it) H(s;\bsalpha+it, \bsbeta+it) 
\Big(\frac{q}{\pi}\Big)^{3s +\delta(\bsalpha,\bsbeta)} \sum_{\substack{{n=1}\\ {(n,q)=1}}}^{\infty} 
\frac{\sigma(n;\bsalpha)\sigma(n;-\bsbeta)}{n^{1+2s}} \frac{ds}{s} dt.  
$$ 
Since  
$$ 
\sum_{\substack{{n=1}\\ {(n,q)=1}}}^{\infty} 
\frac{\sigma(n;\bsalpha)\sigma(n;-\bsbeta)}{n^{1+2s}} 
= \frac{{\mathcal {AZ}}}{\mathcal B_q}(\tfrac 12+s; \bsalpha,\bsbeta), 
$$ 
our expression above equals 
\begin{equation} 
\label{eq:4.7} 
\int_{-\infty}^{\infty} \frac{1}{2\pi i} 
\int_{(1)} G(\tfrac 12+ s;\bsalpha+it,\bsbeta+it) H(s;\bsalpha+it, \bsbeta+it) 
\Big(\frac{q}{\pi}\Big)^{3s +\delta(\bsalpha,\bsbeta)}  \frac{{\mathcal {AZ}}}{\mathcal B_q}(\tfrac 12+s; \bsalpha,\bsbeta) \frac{ds}{s} dt. 
\end{equation}

In the region Re$(s)\ge -\tfrac 14+\epsilon$, the integrand has a simple pole  in $s$ at  $s=0$.   
While ${\mathcal Z}(\tfrac 12+s;\bsalpha,\bsbeta)$ has poles at $s= -(\alpha_j- \beta_\ell)/2$ 
(with $1\le j, \ell \le 3$), these poles are canceled by the zeros of $H(s;\bsalpha+it,\bsbeta+it) 
= H(s;\bsalpha,\bsbeta)$ at these points.   Thus, moving the line of integration in \eqref{eq:4.7} 
to Re$(s)=-\tfrac 14+\epsilon$, we obtain  
$$ 
\int_{-\infty}^{\infty} G(\tfrac 12;\bsalpha+it, \bsbeta+it) H(0;\bsalpha,\bsbeta) 
\frac{\mathcal {AZ}}{\mathcal {B}_q}(\tfrac 12;\bsalpha,\bsbeta) dt +O(q^{-\frac 34+\epsilon}).
$$ 
Using this in  \eqref{eqn:diag1}, we obtain the Lemma.  
 \end{proof}  
 
Thus, $\mathcal D(\Psi,Q;\bsalpha,\bsbeta)$ accounts for one 
of the twenty terms in our conjectured asymptotic formula 
for the shifted sixth moment,  namely the term corresponding to the 
identity permutation in $S_{6}/S_3\times S_3$.  
Similarly ${\mathcal D}(\Psi,Q;\bsbeta,\bsalpha)$ accounts for another 
of these twenty terms, namely the one corresponding to the involution 
$(\bsalpha,\bsbeta) \to (\bsbeta,\bsalpha)$.  
The remaining eighteen expressions arise from off-diagonal terms, 
and we now embark on analyzing their contribution.

\section{Evaluating $\mathcal S(\Psi,Q;\bsalpha,\bsbeta)$}

\noindent  Recall that 
$$ 
{\mathcal S}(\Psi,Q;\bsalpha,\bsbeta) 
= \frac 12 \sum_{\substack{{m,n=1}\\ {m\neq n}}}^{\infty} 
\frac{\sigma(m;\bsalpha)}{\sqrt{m}} \frac{\sigma(n;-\bsbeta)}{\sqrt{n}} 
\sum_{\substack{{d>D, r} \\ {(dr,mn)=1} \\ {r|m\pm n}} }
\mu(d) \phi(r) \Psi\Big(\frac{dr}{Q}\Big) 
V_{\bsalpha,\bsbeta}(m,n;dr). 
$$ 
We express the condition $r|(m\pm n)$ above using the 
characters $\psi \pmod r$.   Thus  the above equals 
\begin{align}
\label{eq:5.1}
 \sum_{ d > D} \sum_{r}\mu(d)
\Psi\Big(\frac{dr}{Q}\Big) 
\sum_{ \substack{{\psi \pmod r}\\ { \psi(-1)=1}}} 
\sum_{ \substack{{m,n =1}\\{ (mn,dr)=1}
 \\ { m\ne n}}}^{\infty} \psi(m)\overline{\psi}(n)
 \frac{\sigma(m;\bsalpha)}{\sqrt{m}} \frac{\sigma(n;-\bsbeta)}{\sqrt{n}}
   V_{\bsalpha,\bsbeta}(m,n;dr). 
\end{align}
We now isolate the contribution of the principal character $\psi =\psi_0$ 
above which gives rise to a main term, and we show that the 
non-principal characters give an acceptable error term.  

\begin{proposition}  We have 
$$ 
{\mathcal S}(\Psi,Q;\bsalpha,\bsbeta) = 
{\mathcal {MS}}(\Psi,Q;\bsalpha,\bsbeta) + O (Q^{2+\epsilon}/D), 
$$ 
where 
$$ 
{\mathcal {MS}}(\Psi,Q;\bsalpha,\bsbeta) = -\sum_q 
\Big( \sum_{\substack {{dr=q}\\ {d\le D}}} \mu(d)\Big) 
\Psi\Big(\frac{q}{Q}\Big) \sum_{\substack{{m, n=1} \\ {(mn,q)=1} \\ {m\neq n }}} \frac{\sigma(m;\bsalpha)}{\sqrt{m}} 
\frac{\sigma(n;-\bsbeta)} {\sqrt{n}} V_{\bsalpha, \bsbeta}(m,n;q).
$$ 
\end{proposition} 
\begin{proof}   The principal character $\psi=\psi_0$ contributes to 
\eqref{eq:5.1}  the amount
$$ 
\sum_{q } \Big(\sum_{\substack {{dr=q}\\ {d>D} }} 
\mu(d) \Big) \Psi\Big( \frac{q}{Q}\Big) 
\sum_{\substack{{m, n=1} \\ {(mn,q)=1} \\ {m\neq n }}} \frac{\sigma(m;\bsalpha)}{\sqrt{m}} 
\frac{\sigma(n;-\bsbeta)} {\sqrt{n}} V_{\bsalpha, \bsbeta}(m,n;q),  
$$
and since $\sum_{dr=q} \mu(d) =0$ for $q>1$ this equals the term $\mathcal{MS}(\Psi,Q;\bsalpha,\bsbeta)$ identified in our Proposition. 

Now we consider the contribution of the non-principal characters to \eqref{eq:5.1}.   
To estimate this, we begin by reincorporating the terms $m=n$ to \eqref{eq:5.1}.  The 
error introduced in doing this is 
\begin{align*}
&\ll \sum_{\substack{{q=dr}\\ {d>D}}} r \Psi\Big(\frac{dr}{Q}\Big) \sum_{n} 
\frac{|\sigma(n;\bsalpha)\sigma(n;-\bsbeta)|}{n} |V_{\bsalpha,\bsbeta}(n,n;q) |\\
&\ll Q^{\epsilon} \sum_{d>D} \sum_{r} \Psi\Big(\frac{dr}{q}\Big) \ll Q^{\epsilon} 
\sum_{d>D} \frac{Q^2}{d^2} \ll \frac{Q^{2+\epsilon}}{D}, 
\end{align*} 
which is acceptable for our Proposition.  

We now wish to estimate 
\begin{equation} 
\label{eq:5.2} 
\sum_{d>D} \sum_{r} \mu(d) \Psi\Big(\frac{dr}{Q}\Big) 
\sum_{\substack{{\psi \pmod r}\\ {\psi(-1)=1} \\ {\psi\neq \psi_0}}} 
\sum_{\substack{{m, n=1}\\ {(mn,dr)=1}}}^{\infty} \psi(m) 
\overline{\psi(n)} \frac{\sigma(m;\bsalpha)}{\sqrt{m}} \frac{\sigma(n;-\bsbeta)}{\sqrt{n}}
V_{\bsalpha,\bsbeta}(m,n;dr). 
\end{equation} 
From the definition of $V_{\bsalpha,\bsbeta}$ we see that the 
sum over $m$ and $n$ above is 
\begin{align}
\label{eq:5.3} 
\int_{-\infty}^{\infty} \frac{1}{2\pi i} 
\int_{(1)} G(\tfrac 12+s;\bsalpha+it,\bsbeta+it)& H(s;\bsalpha,\bsbeta) \Big(\frac{dr}{\pi}\Big)^{3s+\delta(\bsalpha,\bsbeta)} 
\nonumber \\
&\times \sum_{\substack{{m,n=1}\\ {(mn,dr)=1}}}^{\infty} \frac{\sigma(m;\bsalpha)\sigma(n;-\bsbeta)}{(mn)^{\frac 12+s}} \psi(m)
\overline{\psi(n)} \frac{ds}{s} dt.  
\end{align} 
Now the sum over $m$ and $n$ above gives 
$$ 
\prod_{j=1}^{3 }\prod_{\ell=1}^3 \frac{L(\tfrac 12+\alpha_j+s,\psi)}{L_{dr}(\tfrac 12+\alpha_j+s,\psi)} 
\frac{L(\tfrac 12-\beta_\ell +s,\overline{\psi})}{L_{dr}(\tfrac 12-\beta_\ell+s,\overline{\psi})},
$$ 
where $L_{dr}(s,\psi)=\prod_{p\mid dr}(1-\psi(p)p^{-s})^{-1} $. 
We use this above, and since $\psi$ is non-principal we 
may move the line of integration to Re$(s) = \epsilon >0$ without 
encountering any poles.   Since $G(\tfrac 12+s;\bsalpha+it, \bsbeta+it)$ decreases 
exponentially in $|t|+|\text{Im}(s)|$, and since the finite Euler products in $L_{dr}$ are $\ll Q^{\epsilon}$, 
we find that the quantity in \eqref{eq:5.3} is 
$$ 
\ll Q^{\epsilon} \int_{(\epsilon)} \exp(-|\text{Im}(s)|) \Big( \sum_{j=1}^{3} |L(\tfrac 12+\alpha_j+s,\psi)|^6 
+\sum_{\ell=1}^{3} | L(\tfrac 12-\beta_\ell +s,\overline{\psi})|^6 \Big) |ds|.
$$ 
Using this in \eqref{eq:5.2} we obtain that the quantity there is 
$$ 
\ll Q^{\epsilon} \sum_{r\le 2Q/D} \frac{Q}{r} \sum_{\substack{{\psi \pmod r}\\ {\psi \neq \psi_0}}} 
\int_{(\epsilon)} \exp(-|\text{Im}(s)|) \Big( \sum_{j=1}^{3} |L(\tfrac 12+\alpha_j+s,\psi)|^6 
+\sum_{\ell=1}^{3} | L(\tfrac 12-\beta_\ell +s,\overline{\psi})|^6 \Big) |ds|.
$$ 
Using the large sieve we see that the above is $\ll Q^{2+\epsilon}/D$, 
completing our proof. 
 \end{proof}

\section{Evaluating $\mathcal G(\Psi,Q;\bsalpha,\bsbeta)$:  Switching to the complementary divisor}

\noindent We now begin our treatment of ${\mathcal G}(\Psi,Q;\bsalpha,\bsbeta)$ 
which is the most difficult term in our analysis.   Recall that 
\begin{equation} \label{eqn:G}
\mathcal G(\Psi,Q;\bsalpha,\bsbeta)  =\frac 12  \sum_{\substack{ {m,n =1}\\ {m\neq n} }}^{\infty}
\frac{\sigma(m;\bsalpha)\sigma(n;-\bsbeta)}{\sqrt{mn}}
 \sum_{ {{d, r \atop (dr,mn)=1} \atop r|m\pm n}\atop d\le D}
 \phi(r)\mu(d) \Psi\Big(\frac{dr}Q\Big) V_{\bsalpha,\bsbeta}(m,n;dr).
 \end{equation}
 
 To proceed further, we write $g=(m,n)$ (the gcd of $m$ and $n$), and 
 put $m=gM$ and $n=gN$ so that $(M,N)=1$.   For a given $m$, $n$, 
 the inner sum over $d$ and $r$ above may be written as 
 \begin{equation}
 \label{eq:6.2}
 \sum_{\substack{{d\le D} \\ {(d,gMN)=1} } } 
 \mu(d) \sum_{ \substack{{r| (M\pm N)} \\{ (r,gMN)=1}} } \phi(r) \Psi\Big(\frac{dr}{Q} \Big) 
 V_{\bsalpha,\bsbeta}(gM,gN;dr). 
 \end{equation} 
 We may express the condition that $r|(M\pm N)$ by writing $|M\pm N| = rh$, and 
 our goal now is to eliminate $r$ from the sum above, recasting it in terms of the 
 complementary divisor $h$.  The reason for doing this is that while $r$ is 
 large in this term, the complementary divisor $h$ is small.  
 
 First we rewrite the arithmetic function $\phi(r)$ as $\sum_{a\ell=r} \mu(a) \ell$.  Then note that 
 the  condition $(r,gMN)=1$ is equivalent to $(r,g)=1$ (since $(M,N)=1$ and $r|(M\pm N)$) and 
 this is equivalent to $(a,g)=1$ and $(\ell,g)=1$.   Thus 
 our expression \eqref{eq:6.2} may be recast as 
 $$ 
 \sum_{\substack{{d\le D} \\ {(d,gMN)=1} } } 
 \mu(d) \sum_{ (a,g)=1
} \mu(a) \sum_{ \substack{{\ell} \\ {|M\pm N| = a\ell h }\\{ (\ell, g)=1}}} \ell
\Psi\Big(\frac{da\ell}{Q}\Big)V_{\bsalpha,\bsbeta}(gM,gN;da\ell).
$$ 
Now we express the condition that $(\ell,g)=1$ using $\sum_{b|(\ell,g)} \mu(b)$.  Thus the 
above equals 
\begin{equation*}
 \sum_{\substack{{ d \le D} \\{ (d,gMN)=1}}} \mu(d)
  \sum_{ (a,g)=1} \mu(a) \sum_{ b |
g} \mu(b) \sum_{\substack{{k}\\{ |M\pm N|= ab k h}}}  b k 
\Psi\Big(\frac{dab k}{Q}\Big) V_{\bsalpha,\bsbeta}(gM,gN;dab k).
 \end{equation*}
 At this juncture, we may replace the sum over $k$ by a sum over the 
 complementary variable $h$ with the condition that 
 $M\pm N \equiv 0 \pmod{abh}$, and then we may write $|M\pm N|/(abh)$ in place of $k$.  
 Thus we may eliminate the variable $k$, and recast the above as
 \begin{align}\label{eq:6.3} 
\sum_{ \substack{{d \le D }\\ { (d,gMN)=1}}} 
 \sum_{(a,g)=1} \sum_{ b | g} 
 \sum_{ \substack{{h >0}\\  { M\equiv \mp N \pmod{ab h}}}} 
& \mu(d)\mu(a)\mu(b)\nonumber \\
 &\times \bigg(\frac{|M\pm N|}{a h}\bigg)
  \Psi\Big(\frac{|M\pm N| d}{Qh}\Big)
V_{\bsalpha,\bsbeta}\Big(gM,gM;\frac{d|M\pm N|}h\Big).
\end{align}

To simplify this expression we define, for 
non-negative real numbers $u$, $x$ and $y$, and 
for each choice of sign,
\begin{equation} 
\label{eq:6.4} 
{\mathcal W}^{\pm}_{\bsalpha,\bsbeta}(x,y;u) 
= u|x\pm y| \Psi(u|x\pm y|) V_{\bsalpha,\bsbeta}(x,y;u|x\pm y|). 
\end{equation} 
Recall from \eqref{eqn:3} (and with $k=3$ there) 
that 
$$ 
V_{\bsalpha,\bsbeta}(\xi,\eta;\mu) = \Big( \frac{\mu}{\pi}\Big)^{\delta(\bsalpha,\bsbeta)} 
\int_{-\infty}^{\infty} \Big(\frac{\eta}{\xi}\Big)^{it} 
W\Big(\frac{\xi \eta \pi^3}{\mu^3};\bsalpha+it,\bsbeta+it\Big) dt, 
$$ 
and so, for $c>0$,
\begin{eqnarray*}
 c^{-2\delta(\bsalpha,\bsbeta)/3}V_{\bsalpha,\bsbeta}(c\xi,c\eta;\mu c^{2/3})=
V_{\bsalpha,\bsbeta}(\xi,\eta;\mu).
\end{eqnarray*}

With this notation the quantity in \eqref{eq:6.3} becomes 
\begin{equation} 
\label{eq:6.5} 
Q^{1+\delta(\bsalpha,\bsbeta)}\sum_{\substack{{d \le D }\\ { (d,gMN)=1}}} 
 \sum_{(a,g)=1} \sum_{ b | g} 
 \sum_{ \substack{{h >0}\\  { M\equiv \mp N \pmod{ab h}}}} 
 \frac{\mu(a)\mu(b)\mu(d)}{ad} {\mathcal W}^{\pm}_{\bsalpha,\bsbeta}\Big(\frac{gM}{Q^{\frac 32}},\frac{gN}{Q^{\frac 32}};\frac{Q^{\frac 12}d}{gh}\Big).
 \end{equation}
We now express the condition $M\equiv \mp N \pmod{abh}$ using
characters $\psi \pmod{abh}$; note that since $(M,N)=1$ we have $(MN,abh)=1$.  We isolate the contribution  of the
principal character which gives rise to main terms, 
and the remaining characters will contribute an acceptable error term.  

Putting the above remarks together, we have shown that 
 \begin{equation} 
 \label{eq:6.6} 
 {\mathcal G}(\Psi,Q;\bsalpha,\bsbeta) = {\mathcal {MG}}(\Psi,Q;\bsalpha,\bsbeta) + 
 {\mathcal {EG}}(\Psi,Q;\bsalpha,\bsbeta), 
 \end{equation} 
 where the main term is 
 \begin{align} 
 \label{eq:6.7} 
 {\mathcal {MG}}(\Psi,Q;\bsalpha,\bsbeta) &= \frac {Q^{1+\delta(\bsalpha,\bsbeta)}}2  
 \sum_{\substack{ {m,n =1}\\ {m\neq n} }}^{\infty}
\frac{\sigma(m;\bsalpha)\sigma(n;-\bsbeta)}{\sqrt{mn}} \sum_{\substack{{d \le D }\\ { (d,gMN)=1}}} 
   \sum_{(a,g)=1} \sum_{ b | g} 
 \sum_{ \substack{{h >0}\\  { (abh,MN)=1}}} \nonumber\\
&\hskip 1 in  \times\frac{\mu(a)\mu(b)\mu(d)}{ad\phi(abh)} 
 {\mathcal W}^{\pm}_{\bsalpha,\bsbeta}\Big( \frac{gM}{Q^{\frac 32}}, \frac{gN}{Q^{\frac 32}}; 
 \frac{Q^{\frac 12}d}{gh}\Big), 
\end{align}
and the error term is 
\begin{align}
\label{eq:6.8}
{\mathcal {EG}}(\Psi,Q;\bsalpha,\bsbeta) &= \frac {Q^{1+\delta(\bsalpha,\bsbeta)}}2  
 \sum_{\substack{ {m,n =1}\\ {m\neq n} }}^{\infty}
\frac{\sigma(m;\bsalpha)\sigma(n;-\bsbeta)}{\sqrt{mn}} \sum_{\substack{{d \le D }\\ { (d,gMN)=1}}} 
   \sum_{(a,g)=1} \sum_{ b | g} 
 \sum_{ \substack{{h >0}\\  { (abh,MN)=1}}} \nonumber\\
&\hskip 0.5in  \times\frac{\mu(a)\mu(b)\mu(d)}{ad\phi(abh)} 
\sum_{\substack{{\psi\pmod{abh}}\\ {\psi \neq\psi_0}}} 
\psi(M) \overline{\psi}(\mp N) 
{\mathcal W}^{\pm}_{\bsalpha,\bsbeta}\Big( \frac{gM}{Q^{\frac 32}}, \frac{gN}{Q^{\frac 32}}; 
 \frac{Q^{\frac 12}d}{gh}\Big).
 \end{align}

\section{ Mellin transforms of the weight function}
  
 In this section we consider three types of Mellin transforms of 
 the weight functions ${\mathcal W}_{\bsalpha,\bsbeta}^{\pm}(x,y;u)$ 
 which we shall find useful.   The three types arise by taking 
 Mellin transforms in (a) the variable $u$, (b) the variables $x$ and $y$, and (c) 
 all three variables $x$, $y$, and $u$.  We now consider each of these 
 in order, and establish some of their key properties.   
 
 Let us begin by noting that 
${\mathcal W}_{\bsalpha,\bsbeta}^\pm(x,y;u)$ is smooth in $u$, $x$ and $y$ for 
either choice of sign.   When the sign is $+$  this is clear from the 
definition in \eqref{eq:6.4}, but one may worry about the 
$|x-y|$ term when the sign is $-$.   This does not cause trouble 
because $\Psi$ is smooth and supported away from zero.  

%Moreover ${\mathcal W}_{\bsalpha,\bsbeta}^\pm$ and its derivatives are small unless $x$
%and $y$ are small with $x/y \asymp 1$.   In particular,   for any $C_1,C_2,C_3>0$, we have
%\begin{eqnarray} \label{eqn:West1}
%{\mathcal W}_{\bsalpha,\bsbeta}^\pm(x,y;u)\ll u^{C_1}x^{-C_2}y^{-C_3}.
%\end{eqnarray}
%This estimate easily follows from the representation
%\begin{eqnarray*}
%\mathcal W^{\pm}_{\bsalpha,\bsbeta}=u|x\pm y|\Psi(u|x\pm y|)(\frac u\pi)^\delta
%\int_{-\infty}^\infty\Phi(t)(\frac x y)^{-it}\frac{1}{2\pi i}\int_s G_{A_{s+it},B_{s-it}}(\frac{xy\pi^K}{(u|x\pm y|)^K})^{-s}
%\frac{H(s)}{s}~ds~dt,
%\end{eqnarray*}
%by considering a variety of cases. Especially useful is the fact that $\mathcal W^\pm_{A,B}(u,x,y)=0$ unless $1\le u|x\pm y|<2$. A less obvious case is when $x$ is very small, $y$ is large, and $u\approx 1/y$. One moves the path of
%integration in $s$ far to the right, and then the path of integration of $t$ so that $it$ is the vertical line 
%just slightly to the left of the $s$-line. Then the power on $x$ in the new integral is small, the power on $y$ is large and negative, and since $u$ is about $1/y$, the estimate (\ref{eqn:West1}) follows.

 First we define 
 \begin{equation} 
 \label{eq:7.1} 
{\widetilde {\mathcal W}}_1^{\pm}(x,y;z ) = \int_{0}^{\infty} {\mathcal W}_{\bsalpha,\bsbeta}^{\pm} (x,y;u) u^{z} 
 \frac{du}{u}.
 \end{equation} 
 These transforms will be used in \S 9 in simplifying the contribution of 
 ${\mathcal MS}(\Psi,Q;\bsalpha,\bsbeta) + {\mathcal MG}(\Psi,Q;\bsalpha,\bsbeta)$.  
 
 \begin{lemma}
 \label{lemW1}   Given positive real numbers $x$ and $y$, the functions ${\widetilde {\mathcal W}}_1^{\pm}(x,y;z)$ are  analytic for all $z\in {\Bbb C}$.  We have the Mellin inversion formula 
 \begin{equation} 
 \label{eq:7.2}
 {\mathcal W}_{\bsalpha,\bsbeta}^{\pm}(x,y;u) = \frac{1}{2\pi i} \int_{(c)} {\widetilde {\mathcal W}}_1^{\pm}(x,y;z )
  u^{-z} dz,
  \end{equation} 
  where the integral is taken over the line Re$(z)=c$, for any real number $c$.   
  The Mellin transforms ${\widetilde {\mathcal W}}_1^{\pm}(x,y;z)$ satisfy for any 
  non-negative integer $\nu$  
  \begin{equation}
  \label{eq7.21} 
  |\widetilde {\mathcal W}_1^{\pm}(x,y;z)| \ll_\nu |x\pm y|^{-\text{Re }z} \prod_{j=1}^{\nu} |z+j|^{-1} 
\exp\Big(- c_1 \max(x,y)^{\frac 13}\Big) 
\end{equation} 
for some absolute constant $c_1$.  
 \end{lemma} 
 \begin{proof}  From the definition we have 
 $$ 
 {\widetilde {\mathcal W}_1^{\pm}}(x,y;z) = \int_0^{\infty} 
 u|x\pm y |\Psi(u|x\pm y|) V_{\bsalpha,\bsbeta}(x,y;u|x\pm y| ) u^z \frac{du}{u}. 
 $$ 
 When $x\pm y \neq 0$ (which happens always if the choice of sign is $+$) we 
 may make a change of variables $w=u|x\pm y|$, and then the 
 above becomes 
 \begin{equation} 
 \label{eq:7.22}
 |x\pm y|^{-z} \int_0^{\infty} \Psi(w) V_{\bsalpha,\bsbeta}(x,y;w) w^{z} dw.
 \end{equation} 
 Since $\Psi$ is compactly supported away from zero, the above expression 
 gives an analytic function of $z$ for all $z\in {\Bbb C}$.  
 Note that when $x-y=0$ the transform ${\widetilde {\mathcal W}}_1^{-}(x,x;z)$ is 
 identically zero.  
 
 The Mellin inversion formula \eqref{eq:7.2} is standard.  Finally, the estimate \eqref{eq7.21} 
 follows upon integrating the expression in \eqref{eq:7.22} by parts $\nu$ times, using 
 Lemma  \ref{weights} to bound the derivatives of $V_{\bsalpha,\bsbeta}$.  
 \end{proof}
 %Using here the definition of $V_{\bsalpha,\bsbeta}$ we find that the 
 %above equals 
 %$$ 
 %|x\pm y|^{-z} \int_0^{\infty} \Psi(w) \Big( \frac{w}{\pi}\Big)^{\delta(\bsalpha,\bsbeta)} 
% w^z \int_{-\infty}^{\infty} \Big(\frac{y}{x}\Big)^{it} 
 %\frac{1}{2\pi i} \int_{(1)} G(\tfrac 12+s;\bsalpha+it,\bsbeta+it) H(s;\bsalpha,\bsbeta) 
 %\Big(\frac{xy\pi^3}{w^3}\Big)^{-s} \frac{ds}{s} dt dw. 
 %$$ 
% Performing first the integral over $w$, we find that this equals
 %$$ 
 %|x\pm y|^{-z} \pi^{-\delta(\bsalpha,\bsbeta)} 
 %\int_{-\infty}^{\infty} \Big( \frac {y}{x}\Big)^{it} \frac{1}{2\pi i} 
 %\int_{(1)} 
 %G(\tfrac 12+s;\bsalpha+it,\bsbeta+it) H(s;\bsalpha,\bsbeta) 
 %(xy\pi^3)^{-s} {\widetilde \Psi}(1+ 3s+z+\delta(\bsalpha,\bsbeta)) \frac{ds}{s} dt. 
 %$$ 
 %Since ${\widetilde \Psi}(1+3s+z+\delta(\bsalpha,\bsbeta))$ is analytic in $z$ 
 %for all $z\in {\Bbb C}$, the result follows. 

 Next we define 
 \begin{equation}
 \label{eq:7.3} 
 {\widetilde {\mathcal W}}_2^{\pm}(s_1, s_2 ; u) = \int_0^{\infty} \int_0^{\infty} 
 {\mathcal W}_{\bsalpha,\bsbeta}^{\pm}(x,y;u) x^{s_1} y^{s_2} \frac{dx}{x} \frac{dy}{y}. 
 \end{equation} 
 These transforms will be used in \S 8 to estimate the error terms in \eqref{eq:6.8}.  
 
 \begin{lemma} 
 \label{lemW2}  Given a positive real number $u$, the 
 functions ${\widetilde{ \mathcal W}_2}^{\pm}(s_1,s_2; u)$ are  analytic in the 
 region Re$(s_1)$, Re$(s_2)>0$. 
 We have the Mellin inversion formula 
 $$ 
 {\mathcal W}_{\bsalpha,\bsbeta}^{\pm} (x,y;u) = \frac{1}{(2\pi i)^2} \int_{(c_1)} \int_{(c_2)} 
 \widetilde{\mathcal W}_2^{\pm}(s_1, s_2;u) x^{-s_1} y^{-s_2} ds_1 ds_2, 
 $$ 
 where $c_1$ and $c_2$ are positive.   The Mellin transforms $\widetilde{ {\mathcal W}_2}^{\pm} (s_1, s_2; u)$  satisfy, for any $k\ge 1$ 
 $$ 
 |\widetilde{ {\mathcal W}_2}^{\pm} (s_1, s_2; u) | \ll \frac{(1+u)^{k-1}}{\max(|s_1|, |s_2|)^k} \exp(-c_1 u^{-\frac 13}). 
 $$ 
 \end{lemma} 
 \begin{proof}   For a fixed positive number $u$, note that ${\mathcal W}_{\bsalpha,\bsbeta}^{\pm}(x,y;u)$ 
 is zero unless $1\le u|x\pm y|\le 2$, and decreases rapidly as $\max(x,y)$ gets large.  Therefore 
 we see that the Mellin transform $\widetilde {\mathcal W}_2^{\pm}$ is 
 analytic in the region Re$(s_1)$ and Re$(s_2)>0$.   Finally, the estimate on 
 the Mellin transform follows upon integrating by parts $k$ times.
 \end{proof}

 Finally, we define 
 \begin{equation} 
 \label{eq:7.4} 
 {\widetilde {\mathcal W}}_3^\pm(s_1,s_2;z) =  \int_0^{\infty}\int_0^{\infty} \int_0^{\infty}
 {\mathcal W}_{\bsalpha,\bsbeta}^{\pm} (x,y;u) u^z x^{s_1} y^{s_2} \frac{du}{u} \frac{dx}{x} 
 \frac{dy}{y}. 
 \end{equation} 
 Further we set 
 \begin{equation} 
 \label{eq:7.41} 
 {\widetilde {\mathcal W}}_3 ( s_1, s_2; z)  = 
 {\widetilde {\mathcal W}}_3^{+}  (s_1,s_2;z) + {\widetilde {\mathcal W}}_3^{-} (s_1, s_2;z). 
 \end{equation} 
 These transforms will play a crucial role in \S 10 where we complete the evaluation of the 
 main terms.  
 
 \begin{lemma} 
 \label{lemW3}  For brevity, set below $\omega = (s_1+s_2-z)/2$ and $\xi = (s_1-s_2+z)/2$.  
 In the region Re$(s_1)$, Re$(s_2)>0$, and $|\text{Re}(s_1-s_2)| < \text{Re}(z) <1$ 
 we have 
\begin{equation} 
\label{eq:7.42}  
{\widetilde {\mathcal W}}_3^{\pm}(s_1,s_2;z) = {\widetilde \Psi} (1+
\delta (\bsalpha,\bsbeta) +3\omega +z) \frac{H(\omega;\bsalpha,\bsbeta)}{2\omega 
\pi^{\delta(\bsalpha,\bsbeta)+3\omega}} \int_{-\infty}^{\infty} 
{\mathcal H}^{\pm}(\xi-it,z) G(\tfrac 12+\omega;\bsalpha+it ,\bsbeta+it) dt, 
\end{equation} 
where 
\begin{equation}
\label{eq:7.43} 
{\mathcal H}^+ (u,v) = \frac{\Gamma(u) \Gamma(v-u)}{\Gamma(u)}, 
\qquad \text{and} \qquad 
{\mathcal H}^{-}(u,v) = \frac{\Gamma(u)\Gamma(1-v)}{\Gamma(1+u-v)}  + 
\frac{\Gamma(v-u)\Gamma(1-v)}{\Gamma(1-u)}.  
\end{equation} 
Furthermore, we have in the same region of variables
\begin{equation} 
\label{eq:7.44}  
{\widetilde {\mathcal W}}_3 (s_1,s_2;z) = {\widetilde \Psi} (1+
\delta (\bsalpha,\bsbeta) +3\omega +z) \frac{H(\omega;\bsalpha,\bsbeta)}{2\omega 
\pi^{\delta(\bsalpha,\bsbeta)+3\omega}} \int_{-\infty}^{\infty} 
{\mathcal H} (\xi-it,z) G(\tfrac 12+\omega;\bsalpha+it ,\bsbeta+it) dt, 
\end{equation} 
where 
\begin{equation}
\label{eq:7.431} 
{\mathcal H}(u,v) = {\mathcal H}^+ (u,v)  +{\mathcal H}^{-}(u,v) =  \pi^{\frac 12}
\frac{\Gamma(\frac{u}{2})  \Gamma (\frac{1-v}{2}) \Gamma(\frac{v-u}{2})}{\Gamma (\frac{1-u}{2} )
\Gamma(\frac{v}{2}) \Gamma(\frac{1-v+u}{2})}.  
\end{equation} 
We have the Mellin inversion formula 
\begin{eqnarray}\label{eqn:W2}
\mathcal W^\pm _{\bsalpha,\bsbeta}(x,y;u)=\frac{1}{(2\pi
i)^3}\int_{z}\int_{s_1}\int_{s_2}\widetilde{\mathcal
W}_{3}^{\pm}(s_1,s_2;z) u^{-z}x^{-s_1}y^{-s_2}~ds_2 ~ds_1~dz
\end{eqnarray}
where all of the paths are taken to be the vertical lines with
increasing imaginary parts and real parts satisfying the constraints given above, and 
the integrals over $s_1$ and $s_2$ are to be interpreted as being 
over $|\text{Im} (s_1)| \le T_1$ and $|\text{Im}(s_2)| \le T_2$ and 
letting $T_1$, $T_2$ tend to infinity.  
 Finally, the Mellin transform ${\widetilde {\mathcal W}}_3(s_1,s_2;z)$ 
 satisfies the bound 
\begin{equation}
\label{eq:7.45} 
 |{\widetilde {\mathcal W}}_3(s_1,s_2;z)| \ll (1+|z|)^{-A} (1+|\omega|)^{-A} (1+|\xi|)^{\text{Re}(z)-1}. 
\end{equation}
  \end{lemma} 
 \begin{proof}  From the definitions we have 
 $$ 
 \widetilde {\mathcal W}_3^{\pm} (s_1,s_2;z) 
 = \int_0^{\infty}\int_0^{\infty} \int_0^{\infty} u |x\pm y| 
 \Psi(u|x\pm y|) V_{\bsalpha,\bsbeta}(x,y;u|x\pm y|) u^z x^{s_1}y^{s_2} 
 \frac{du}{u} \frac{dx}{x}\frac{dy}{y}. 
 $$ 
 The inner integral over $u$ is the Mellin transform ${\widetilde{ \mathcal W}_1^\pm}$ 
 and so, as in \eqref{eq:7.22}, the above equals 
 $$
 \int_0^{\infty}\int_0^{\infty} \int_0^{\infty} |x\pm y|^{-z} \Psi(w) V_{\bsalpha,\bsbeta}(x,y;w) 
 w^z x^{s_1} y^{s_2} dw \frac{dx}{x} \frac{dy}{y}. 
 $$ 
  When the sign above is $-$, the integrand is singular on the 
  diagonal $x=y$ and for $x$ and $y$ near $0$.  Note that the integral 
  is well defined provided Re$(s_1)$ and Re$(s_2)$ are positive, 
  and so long as Re$(z) < \text{Re }(s_1+s_2)$.

  Keeping to this 
  region of parameters $s_1$, $s_2$, and $z$, let us now consider the integrals 
  over $x$ and $y$, for a fixed value of $w$.  Making the substitution $x=\lambda y$ 
  we obtain that ${\widetilde {\mathcal W}_3^{\pm}}$ equals 
  $$ 
  \int_0^{\infty} \Psi(w) w^z \int_0^{\infty} |\lambda \pm 1 |^{-z} \lambda^{s_1} 
  \int_0^{\infty} V_{\bsalpha,\bsbeta}(\lambda y,y;w) y^{s_1+s_2-z}\frac{dy}{y} 
  \frac{d\lambda}{\lambda} dw.
  $$
Employing now the notation $\omega=(s_1+s_2-z)/2$ and $\xi =( s_1-s_2+z)/2$, 
  by (\ref{eqn:W}) and (\ref{eqn:3}) we have
 \begin{equation*}
 V_{\bsalpha,\bsbeta}(\lambda y,y;w)=\Big(\frac w \pi \Big)^{\delta(\bsalpha,\bsbeta)} 
  \int_{-\infty}^{\infty}  \lambda^{-it}
  \frac{1}{2\pi i}\int_{(1)} G(\tfrac 12+s;\bsalpha+it,\bsbeta+it) H(s;\bsalpha,\bsbeta) 
  \Big(\frac{\lambda y^2\pi^3}{w^3}\Big)^{-s}\frac{ds}{s} dt
 \end{equation*}
so that, by Mellin inversion,
\begin{eqnarray*}
\int_0^\infty V_{\bsalpha,\bsbeta}(\lambda y,y;w)y^{s_1+s_2-z}\frac{dy}{y}
=\Big(\frac w \pi\Big)^{\delta(\bsalpha,\bsbeta)+3\omega} \lambda^{-\omega}\frac{H(\omega;\bsalpha,\bsbeta)}{2\omega}
  \int_{-\infty}^{\infty}  \lambda^{-it}
   G(\tfrac 12+ \omega;\bsalpha+it,\bsbeta+it)~dt.
 \end{eqnarray*}
Our work thus far gives, with $\widetilde \Psi(s) =\int_0^{\infty} \Psi(w)w^{s}dw/w$, 
\begin{align} 
\label{eq:7.5} 
{\widetilde {\mathcal W}_3^{\pm} }(s_1,s_2;z) 
&= 
\frac{H(\omega;\bsalpha,\bsbeta)}{2\omega\pi^{\delta(\bsalpha,\bsbeta)+3\omega} } 
\widetilde{\Psi}(1+z+\delta(\bsalpha,\bsbeta)+3\omega)
\nonumber\\ 
&\hskip 1 in \times \int_0^\infty |\lambda\pm 1|^{-z} \lambda^{s_1 -\omega} \int_{-\infty}^{\infty}\lambda^{-it} G(\tfrac 12+\omega;\bsalpha+it,\bsbeta+it) dt \frac{d\lambda}{\lambda}.
\end{align}
% $$ 
 %\pi^{-\delta(\bsalpha,\bsbeta)} 
 %\int_{-\infty}^\infty 
 %\frac{1}{2\pi i} \int_{(1)} 
% G(\tfrac 12+s;\bsalpha+it, \bsbeta+it) 
% H(s;\bsalpha,\bsbeta) \pi^{-3s} 
 %{\widetilde \Psi}(1+3s+z+\delta(\bsalpha,\bsbeta)) 
 %\Big(\int_0^{\infty} \int_{0}^{\infty} |x\pm y|^{-z} x^{s_1-s-it} y^{s_2 -s+it} \frac{dx}{x} 
 %\frac{dy}{y}\Big) \frac{ds}{s} dt.
 %$$ 
 % making a substitution $w= u|x\pm y|$ we 
 %find that it equals 
 %$$ 
 %|x\pm y|^{-z} \int_0^{\infty} \Psi(w) V_{\bsalpha,\bsbeta}(x,y;w) w^{z} dw 
 %= |x\pm y|^{-z} \int_0^{\infty} \Psi(w) \Big(\frac{w}{\pi}\Big)^{\delta(\bsalpha,\bsbeta)} 
% w^z  \int_{-\infty}^{\infty} \Big(\frac{y}{x}\Big)^{it} W\Big( \frac{xy\pi^3}{w^3};\bsalpha+it,\bsbeta+it\Big) dt
% dw.
% $$ 
%With some changes of variables we can write our transformed weight%
%function as
%\begin{eqnarray*}
 %{\tilde {\mathcal W}}_{A,B}^\pm(z,s_1,s_2)=\int_0^\infty
% u^{z+1}\Psi(u)\int_0^\infty \lambda^{s_1}|1\pm \lambda|^{-z}\int_0^\infty
%V_{A,B}(\lambda y,y;u)
%y^{s_1+s_2-z}\frac{dy}{y}\frac{d\lambda}\lambda \frac {du}u
%\end{eqnarray*}

We now evaluate the integral over $\lambda$ above using the familiar beta-integral.  
Precisely, for any complex numbers $u$ and $v$ with $0< \text{Re}(u) <\text{Re}(v) <1$ 
we have
\begin{align}
\label{eq:7.6}
\int_0^\infty \lambda ^u |1-\lambda|^{-v} ~\frac{d\lambda}{\lambda}
&= \int_0^1 \lambda ^{u-1}(1-\lambda)^{-v}~d\lambda+\int_0^1
\lambda^{v-u-1}(1-\lambda)^{-v}~d\lambda \nonumber\\
&=\frac{\Gamma(u)\Gamma(1-v)}{\Gamma(1+u-v)}+\frac{\Gamma(v-u)\Gamma(1-v)}{\Gamma(1-u)} 
= {\mathcal H}^{-}(u,v)
\end{align}
and
\begin{align}
\label{eq:7.7}
\int_0^\infty \lambda ^u|1+\lambda|^{-v} ~\frac{d\lambda}{\lambda}
&=\int_1^\infty (\lambda-1)^{u-1} \lambda^{-v} ~ d\lambda\nonumber
\\
&=  \int_0^1\lambda^{v-u-1}(1-\lambda)^{u-1}d\lambda\nonumber\\
&= \frac{\Gamma(v-u)\Gamma(u)}{\Gamma(v)} = {\mathcal H}^{+}(u,v).
\end{align}
These remarks establish the result given in \eqref{eq:7.42}.   The 
expression given in \eqref{eq:7.431} follows upon summing over both 
choices of sign and employing the following remarkable identity: (see Lemma 8.2 of Young \cite{You}) 
%and we also note that this identity features in work in physics see)
\begin{align}
\label{eq:7.7}
\frac{\Gamma(u)\Gamma(1-v)}{\Gamma(1+u-v)}+\frac{\Gamma(v-u)\Gamma(1-v)}{\Gamma(1-u)}
+\frac{\Gamma(v-u)\Gamma(u)}{\Gamma(v)} 
&=\pi^{\frac 12} \frac{\Gamma(\frac{1-v}{2}) \Gamma(\frac u
2 ) \Gamma(\frac{v-u}{2})}{\Gamma( \frac v2) \Gamma(\frac
{1-u}{2})\Gamma(\frac {1-v+u}{2})}.
\end{align} 
 %From the above, we conclude that $\widetilde {\mathcal
% W}_{3}^\pm(s_1,s_2;z)$ equals 
 %\begin{align}
  % \widetilde{\Psi}(1+\delta(\bsalpha,\bsbeta)+3\omega+z)\frac{H(\omega;\bsalpha,\bsbeta)}{2\omega\pi^{\delta(\bsalpha,\bsbeta)+3\omega}}
 %\int_{-\infty}^\infty
 %\mathcal H^\pm(\xi-it,z)    G(\tfrac 12+\omega;\bsalpha+it,\bsbeta+it) dt.
 %\end{align}

Finally, using Stirling's formula in \eqref{eq:7.44} we may obtain the estimate for ${\widetilde {\mathcal W}}_3$ 
given in \eqref{eq:7.45}.  
\end{proof}

\section{Estimating the error term $\mathcal {EG}(\Psi,Q;\bsalpha,\bsbeta)$}

First we dispense with the contribution to \eqref{eq:6.8}  from values of $a$ larger than $2Q$.  
Since $M \neq N$, if $M\equiv \mp N \pmod {abh}$ and $a>2Q$ then 
$$ 
\frac{Q^{\frac 12}d}{gh} \frac{|gM\pm gN|}{Q^{\frac 32}} \ge \frac{dab}{Q} \ge 2, 
$$ 
so that ${\mathcal W}_{\bsalpha,\bsbeta}^{\pm} (gM/Q^{\frac 32},gN/Q^{\frac 32};Q^{\frac 12}d/(gh))=0$ 
for such terms.  Therefore the contribution to \eqref{eq:6.8} from terms 
with $a>2Q$ is 
$$ 
\ll Q^{1+\epsilon} 
\sum_{m,n=1}^{\infty} \frac{1}{(mn)^{\frac 12-\epsilon}} 
\sum_{d\le D} \sum_{a>2Q} \sum_{b|g} \sum_{h>0} 
\frac{1}{ad\phi(abh)}\Big |{\mathcal W}_{\bsalpha,\bsbeta}^{\pm}
\Big(\frac{m}{Q^{\frac 32}}, \frac{n}{Q^{\frac 32}}; \frac{Q^{\frac 12}d}{gh}\Big)\Big|.
$$
Since ${\mathcal W}_{\bsalpha,\bsbeta}^{\pm} (gM/Q^{\frac 32},gN/Q^{\frac 32};Q^{\frac 12}d/(gh))$ 
is small unless $m$ and $n$ are below $Q^{\frac 32+\epsilon}$ 
and $|m\pm n| d/(Qgh) \in [1,2]$, we find that the above is 
$$ 
\ll Q^{1+\epsilon} \sum_{a>2Q} \frac{1}{a^{2-\epsilon}} \sum_{m, n=1}^{\infty} 
\frac{1}{(mn)^{\frac 12-\epsilon}} \exp\Big( -c\Big(\frac{\max(m,n)}{Q^{\frac 32}}\Big)^{\frac13}\Big) 
\ll Q^{\frac 32+\epsilon}.
$$

We may therefore restrict attention to the terms with $a\le 2Q$.   Here
we use the two variable Mellin transform ${\widetilde {\mathcal W}}_2(s_1,s_2;u)$ 
together with Mellin inversion.  Thus the remainder term in \eqref{eq:6.8} may be written as 
\begin{align}
\label{eq:8.1}
\frac{Q^{1+\delta(\bsalpha,\bsbeta)}}{2} 
\sum_{\substack{{a\le 2Q} \\ {b, h>0}}} \sum_{\substack{{\psi \pmod {abh}}\\ {\psi\neq \psi_0}}} &
\sum_{\substack{{g}\\{b|g, (a,g)=1}}}  \sum_{\substack {{d\le D} \\ {(d,g)=1}}} 
\frac{\mu(a)\mu(b)\mu(d)}{adg\phi(abh)} 
  \frac{1}{(2\pi i)^2} \int_{(\frac 12+\epsilon)} \int_{(\frac 12+\epsilon)} 
{\widetilde {\mathcal W}}_2^{\pm} \Big(s_1,s_2; \frac{Q^{\frac 12}d}{gh}\Big) \nonumber \\ 
&\times \Big(\frac{Q^{\frac 32}}{g}\Big)^{s_1+s_2} 
\sum_{\substack {{M,N=1} \\ {M\neq N, (M,N)=1} }}^{\infty} 
\frac{\sigma(gM;\bsalpha)\sigma(gN;-\bsbeta)}{M^{\frac 12+s_1} N^{\frac 12+s_2}} \psi(M) 
\overline{\psi}(\mp N) ds_1 ds_2.
 \end{align} 
 
The inner sum over $M$ and $N$ in \eqref{eq:8.1}  
may be written as 
$$ 
\psi(\mp 1) \Big( \prod_{j=1}^{3} L(\tfrac 12+ s_1 +\alpha_j,\psi) L(\tfrac 12+s_2 -\beta_j,\overline\psi) 
{\mathcal F}(g,\psi;s_1,s_2)- \sigma(g;\bsalpha)\sigma(g;-\bsbeta)\Big),
$$ 
for a suitable function ${\mathcal F}$ which is analytic when Re$(s_1)$ and Re$(s_2)$ are 
$>\epsilon$ and is bounded there by $|g|^{\epsilon}$.  
Since $\psi$ is not principal, the $L$-functions above are non-trivial 
and have no poles.  Thus we may move the lines of integration to Re$(s_1) =\text{Re}(s_2) 
=\epsilon$.  In this way, we find that the quantity in \eqref{eq:8.1} is bounded by 
\begin{align}
\label{eq:8.2} 
&\ll Q^{1+\epsilon} \sum_{\substack{{a\le 2Q} \\ {b, h>0}}} \sum_{\substack{{\psi \pmod {abh}}\\ {\psi\neq \psi_0}}} 
\sum_{\substack{{g}\\{b|g, (a,g)=1}}}  \sum_{\substack {{d\le D} \\ {(d,g)=1}}} 
\frac{1}{adg \phi(abh)}\nonumber \\
&\hskip 0.2 in \times  \int_{(\epsilon)} \int_{(\epsilon)}  
\Big|{\widetilde{ \mathcal W}}_2^{\pm} \Big(s_1, s_2;\frac{ Q^{\frac 12 }d}{gh}\Big)\Big| 
\Big(1+\prod_{j=1}^{3} |L(\tfrac 12+s_1+\alpha_j,\psi)L(\tfrac 12+s_2-\beta_j,\overline{\psi})|\Big) ds_1 ds_2.
\end{align}
 
 Using the bound in Lemma \ref{lemW2} we obtain that for any $k\ge 1$ 
 \begin{align*}
 \sum_{\substack{{g} \\ {b|g}}} \frac{1}{g} 
 \sum_{d\le D} \frac{1}{d}\Big |\widetilde{ \mathcal W}_2^{\pm} \Big(s_1,s_2 ; 
 \frac{Q^{\frac 12}d}{gh}\Big) \Big| &\ll \frac{(1+Q^{\frac 12}D/(bh))^{k-1}}{\max(|s_1|,|s_2|)^k} 
 (\log D) \sum_{b|g} \frac 1g \exp\Big( -c \Big( \frac{gh}{Q^{\frac 12}D}\Big)^{\frac 13}\Big) 
 \\
 &\ll  \frac{(1+Q^{\frac 12}D/(bh))^{k-1}}{\max(|s_1|,|s_2|)^k} \frac{Q^{\epsilon}}{b} 
 \exp\Big( -c \Big( \frac{bh}{Q^{\frac 12}D}\Big)^{\frac 13}\Big).  
\end{align*}
We use this in \eqref{eq:8.2}, and divide the remaining variables $a$, $b$, $h$, $s_1$ and $s_2$  
into dyadic blocks -- say $A\le a < 2A$, $B\le b<2B$, $H\le h<2H$, $S_1 \le |s_1|<2S_1$ 
and $S_2 \le |s_2| <2S_2$.  Consider now the contribution of such a dyadic block 
to \eqref{eq:8.2}.   For any fixed $k\ge 1$ this is, writing $\ell = abh$ 
\begin{align*}
&\ll \frac{Q^{1+\epsilon} (1+Q^{\frac 12}D/(BH))^{k-1}}{A^{2-\epsilon} B^{2-\epsilon} H^{1-\epsilon}
\max(S_1, S_2)^k} \exp \Big( -c\Big(\frac{BH}{Q^{\frac 12}D}\Big)^{\frac 13}\Big) 
\sum_{ABH \le \ell < 8ABH} 
\sum_{\substack{{\psi\pmod \ell}\\ {\psi \neq \psi_0}} } \\
&\times \int_{\substack{{(\epsilon)} \\ {S_1 \le |s_1|<2S_1}} }\int_{\substack{{(\epsilon)}\\ {S_2 
\le |s_2| < 2S_2}}} 
\Big(1+\sum_{j=1}^{3}( |L(\tfrac 12+s_1+\alpha_j,\psi)|^6 +L(\tfrac 12+s_2-\beta_j,\overline{\psi})|^6)\Big) ds_1 ds_2.
\end{align*}
 Using the large sieve we see that the sums and integrals above contribute 
 $$ 
 \ll \Big( A^2 B^2 H^2 S_1S_2 + (ABH \max(S_1,S_2))^{\frac 32} \min(S_1,S_2)\Big)^{1+\epsilon}. 
$$ 
 We take $k=1$ if $\max(S_1, S_2) \le 1+Q^{\frac 12}D/(BH)$ and $k=4$ otherwise.  
 Summing over all dyadic blocks (and keeping in mind that $A\le Q$) we 
 obtain (with a little calculation) a net estimate 
 of 
 $$ 
 \ll Q^{\frac 32+\epsilon} D + Q^{\frac 74+\epsilon} D^{\frac 32}, 
 $$ 
 with the worst case scenario being when $A$, $B$ and $H$ are small and  $\max(S_1,S_2)$ 
 being of size $Q^{\frac 12}D$.

 Thus we have established that 
 \begin{equation} 
 \label{eq:8.3} 
 {\mathcal EG}(\Psi,Q;\bsalpha,\bsbeta) \ll Q^{\frac 74+\epsilon} D^{\frac 32}. 
 \end{equation} 
 
 % Thus, 
 %in the first case we obtain that the contribution of a dyadic block is 
 %$$ 
 %\ll Q^{1+\epsilon}\exp\Big(-c \Big(\frac{BH}{Q^{\frac 12}D}\Big)^{\frac 13}\Big)
  % \Big( A^{\epsilon} B^{\epsilon} H^{1+\epsilon} \max(S_1,S_2) 
% +  A^{-\frac 12+\epsilon} B^{-\frac 12+\epsilon} H^{\frac 12+\epsilon} \max(S_1,S_2)^{\frac 32+\epsilon}
 %\Big) 
 %$$

\section{The main terms: ${\mathcal {MS}}(\Psi,Q;\bsalpha,\bsbeta)+ 
{\mathcal {MG}}(\Psi,Q;\bsalpha,\bsbeta)$}

\noindent We start by simplifying the expression for ${\mathcal {MG}}(\Psi,Q;\bsalpha,\bsbeta)$ 
which is given in \eqref{eq:6.7}.   Put 
\begin{equation} 
\label{eq:9.1}
F(h,g;MN) = \sum_{(a,gMN)=1} \frac{\mu(a)}{a} \sum_{\substack{{b|g} \\ {(b,MN)=1}}} 
\frac{\mu(b)}{\phi(abh)} = 
\sum_{(\ell,MN)=1} \frac{\mu(\ell) (\ell,g)}{\ell\phi(\ell h)}, 
\end{equation} 
where the last identity follows upon grouping $\ell=ab$.  Then ${\mathcal {MG}}(\Psi,Q;\bsalpha,\bsbeta)$
equals
\begin{equation} 
\label{eq:9.2}
\frac{Q^{1+\delta(\bsalpha,\bsbeta)}}{2} 
\sum_{\substack{{m,n=1}\\ {m\neq n}}}^{\infty} 
\frac{\sigma(m;\bsalpha)\sigma(n;-\bsbeta)}{\sqrt{mn}} 
\sum_{\substack{{d\le D}\\ {(d,gMN)=1}}} \sum_{\substack{{h>0}\\ {(h,MN)=1}}} 
\frac{\mu(d)}{d} F(h,g;MN) {\mathcal W}_{\bsalpha,\bsbeta}^{\pm} 
\Big( \frac{gM}{Q^{\frac 32}},\frac{gN}{Q^{\frac 32}};\frac{Q^{\frac 12}d}{gh}\Big). 
\end{equation} 

Now we use the Mellin transform $\widetilde{\mathcal W}_1^{\pm} (x,y;z)$ 
given in \eqref{eq:7.1} together with the Mellin inversion formula and the 
estimates of Lemma \ref{lemW1}.  
Using \eqref{eq:7.2} with $c=-\epsilon <0$, we find that the sum over $h$ in \eqref{eq:9.2} is 
\begin{equation} 
\label{eq:9.5}
\sum_{\substack{{h=1}\\{(h,MN)=1}}}^{\infty} F(h,g;MN) 
\frac{1}{2\pi i} \int_{(-\epsilon)} 
\widetilde {\mathcal W}_{1}^{\pm} \Big(\frac{gM}{Q^{\frac 32}},\frac{gN}{Q^{\frac 32}}; 
z\Big) \Big(\frac{Q^{\frac 12}d}{gh}\Big)^{-z} dz.
\end{equation} 
We interchange the sum and the integral, and this is legitimate 
as the sum over $h$ converges absolutely when the real part of $z$ is negative.  

\begin{lemma} 
\label{lemma9.1} 
If Re$(s)>0$ we have 
$$ 
\sum_{\substack{{h=1}\\{(h,MN)=1} }}^{\infty} \frac{F(h,g;MN)}{h^s} 
= \zeta(s+1) {\mathcal K}(s;g,MN)
$$ 
where, with $\phi(\ell,s)=\prod_{p|\ell}(1-1/p^s)$, 
$$ 
{\mathcal K}(s;g,MN) = \phi(MN,s+1) \prod_{p\nmid gmN} \Big(1-\frac{1}{p(p-1)} 
+\frac{1}{p^{1+s}(p-1)} \Big) \prod_{\substack{{p|g}\\ {p\nmid MN}} }
\Big(1-\frac{1}{p^{1+s}}  - \frac{1}{p-1} \Big(1-\frac 1{p^s}\Big)\Big).
$$ 
\end{lemma} 
\begin{proof}  This is a straightforward verification.  
\end{proof}

Using Lemma \ref{lemma9.1} in \eqref{eq:9.5} we obtain that the 
quantity there equals 
\begin{equation} 
\label{eq:9.6}
\frac{1}{2\pi i} \int_{(-\epsilon)} \widetilde{\mathcal W}_{1}^{\pm} 
\Big( \frac{gM}{Q^{\frac 32}}, \frac{gN}{Q^{\frac 32}}; z\Big) \zeta(1-z) {\mathcal K}(-z;g,MN) 
\Big(\frac{Q^{\frac 12}d}{g}\Big)^{-z} dz. 
\end{equation}
We now move the line of integration to the right, to the line Re$(z)=\epsilon>0$.  We 
encounter a pole at $z=0$, whose residue (taking into account that the  
contour is oriented clockwise, and that the residue of $\zeta(1-z)$ at $z=0$ is $-1$) equals 
\begin{align} \label{eq:9.7}
  \widetilde{\mathcal W}_{1}^{\pm} 
\Big( \frac{gM}{Q^{\frac 32}}, \frac{gN}{Q^{\frac 32}}; 0\Big) {\mathcal K}(0;g,MN) 
&= \frac{\phi(gMN)}{gMN} \int_{0}^{\infty} {\mathcal W}_{\bsalpha,\bsbeta}^{\pm} 
\Big( \frac{gM}{Q^{\frac 32}}, \frac{gN}{Q^{\frac 32}}; u\Big) \frac{du}{u}\nonumber\\
&= \frac{\phi(gMN) }{gMN}\int_0^{\infty} u |x\pm y| \Psi(u|x\pm y|) V_{\bsalpha,\bsbeta}(x,y;u|x\pm y|) 
\frac{du}{u} \nonumber \\ 
&=\frac{\phi(gMN)}{gMN} \int_0^{\infty} \Psi(u) V_{\bsalpha,\bsbeta}\Big(\frac{gM}{Q^{\frac 32}}, \frac{gN}{Q^{\frac 32}} ; u\Big) du,
\end{align}
where in the middle line we wrote $x=gM/Q^{\frac 32}$ and $y=gN/Q^{\frac 32}$ for brevity. 
Thus the quantity in \eqref{eq:9.6} equals the residue term above, 
together with the remaining integral on the line Re$(z)=\epsilon$.

Consider first the contribution of the residue term \eqref{eq:9.7} 
to ${\mathcal {MG}}(\Psi,Q;\bsalpha,\bsbeta)$ -- we shall show that 
this contribution cancels the contribution of $\mathcal {MS}(\Psi,Q;\bsalpha,\bsbeta)$.  
Keeping in mind that the sum in \eqref{eq:9.2} is over both choices of 
sign $\pm$, and since $\phi(gMN)/(gMN) = \phi(mn)/(mn)$, 
we see that the residue term \eqref{eq:9.7} contributes 
\begin{equation} 
\label{eq:9.8}
Q^{1+\delta(\bsalpha,\bsbeta)} \sum_{\substack{{m, n=1}\\ {m\neq n}}}^{\infty} 
\frac{\sigma(m;\bsalpha)\sigma(n;-\bsbeta)}{\sqrt{mn}} 
\sum_{\substack{{d\le D} \\ {(d,mn)=1}}} \frac{\mu(d)}{d} \frac{\phi(mn)}{mn} 
\int_0^{\infty} \Psi(u) V_{\bsalpha,\bsbeta}\Big( \frac{m}{Q^{\frac 32}}, 
\frac{n}{Q^{\frac 32}};u \Big) du. 
\end{equation} 
Now consider $\mathcal {MS}(\Psi,Q;\bsalpha,\bsbeta)$ which we recall from Proposition 3 
equals 
$$ 
- \sum_{\substack{{m,n=1}\\ {m\neq n}}} \frac{\sigma(m;\bsalpha)\sigma(n;-\bsbeta)}{\sqrt{mn}} 
\sum_{\substack{{d\le D}\\ {(d,mn)=1}}} \mu(d) \sum_{\substack{{r}\\ {(r,mn)=1}}} 
\Psi\Big( \frac{dr}{Q}\Big) V_{\bsalpha,\bsbeta}(m,n;dr). 
$$ 
Since 
$$ 
\sum_{\substack{{r\le x}\\ {(r,mn)=1}}} 1 = \frac{\phi(mn)}{mn} x +O( (mn)^{\epsilon}),
$$ 
and $V_{\bsalpha,\bsbeta}$ is small unless $mn \le Q^{3+\epsilon}$ we 
find that the above is 
\begin{equation*} 
  - Q \sum_{\substack{{m,n=1}\\ {m\neq n}}}^{\infty} 
  \frac{\sigma(m;\bsalpha)\sigma(n;-\bsbeta)}{\sqrt{mn}} 
  \sum_{\substack{{d\le D}\\ {(d,mn)=1}} } \frac{\mu(d)}{d} 
  \frac{\phi(mn)}{mn} \int_0^\infty \Psi(u)
V_{\bsalpha,\bsbeta}(m,n;uQ) du+O(DQ^{\frac 32+\epsilon}).
\end{equation*}
Since $V_{\bsalpha,\bsbeta}(m,n;uQ) = Q^{\delta(\bsalpha,\bsbeta)} V_{\bsalpha,\bsbeta}(mQ^{-\frac 32}, nQ^{-\frac 32};u)$ the main term above exactly cancels the quantity in 
\eqref{eq:9.8}!

 Summarizing our discussions above we have established that ${\mathcal {MS}}(\Psi,Q;\bsalpha,\bsbeta) +  {\mathcal {MG}}(\Psi,Q;\bsalpha,\bsbeta)$ equals, up to an error $O(D Q^{\frac 32+\epsilon})$, 
 \begin{align}
 \label{eq:9.9}
& \frac{ Q^{1+\delta(\bsalpha,\bsbeta)}}{2} 
 \sum_{\substack{{m,n=1}\\ {m\neq n}}}^{\infty} 
 \frac{\sigma(m;\bsalpha)\sigma(n;-\bsbeta)}{\sqrt{mn} }\nonumber \\
 &\hskip 1 in \times 
 \frac{1}{2\pi i} 
 \int_{(\epsilon)} \widetilde{\mathcal W}_{1}^{\pm} 
 \Big( \frac{gM}{Q^{\frac 32}},\frac{gN}{Q^{\frac 32}};z\Big) 
 \zeta(1-z) {\mathcal K}(-z;g,MN) \Big( \frac{Q^{\frac 12} }{g}\Big)^{-z} \sum_{\substack {{d\le D}\\ {(d,gMN)=1}}} \frac{\mu(d)}{d^{1+z}} dz. 
 \end{align}
 
We move the line of integration above to Re$(z)=1-\epsilon$, and 
then extend the sum over $d$ to include all values of $d$.  Using Lemma \ref{lemW1} the error in 
extending our sum over $d$ is 
\begin{align*}
&\ll Q^{1+\epsilon} \sum_{\substack{ {m,n=1} \\ {m\neq n} } }^{\infty} 
\frac{1}{(mn)^{\frac 12-\epsilon}} 
\int_{(1-\epsilon)} \Big|\frac{m\pm n}{Q^{\frac 32}}\Big|^{-1+\epsilon} 
|z|^{-10} \exp(- c_1 (\max(m,n)/Q^{\frac 32})^{\frac 13}) \Big(\frac{Q^{\frac 12}}{g}\Big)^{
-1+\epsilon} D^{-1+\epsilon} |dz|
 \\
 &\ll Q^{2+\epsilon} D^{-1}.\\
 \end{align*}
 Moving now the line of integration back to Re$(z)=\epsilon$, we conclude that with an error $O(Q^{2+\epsilon}D^{-1}+DQ^{\frac 32+\epsilon})$ 
the quantity ${\mathcal {MS}}(\Psi,Q;\bsalpha,\bsbeta) +  {\mathcal {MG}}(\Psi,Q;\bsalpha,\bsbeta)$ equals 
\begin{align}
 \label{eq:9.10}
& \frac{ Q^{1+\delta(\bsalpha,\bsbeta)}}{2} 
 \sum_{\substack{{m,n=1}}}^{\infty} 
 \frac{\sigma(m;\bsalpha)\sigma(n;-\bsbeta)}{\sqrt{mn} }\nonumber \\
 &\hskip 1 in \times 
 \frac{1}{2\pi i} 
 \int_{(\epsilon)} \widetilde{\mathcal W}_{1}^{\pm} 
 \Big( \frac{gM}{Q^{\frac 32}},\frac{gN}{Q^{\frac 32}};z\Big) 
 \frac{\zeta(1-z) {\mathcal K}(-z;g,MN)}{\zeta(1+z) \phi(gMN,1+z)} \Big( \frac{Q^{\frac 12} }{g}\Big)^{-z} \  dz. 
 \end{align}
Note that above we reintroduced the terms $m=n$ with an acceptable error of $O(Q^{1+\epsilon})$.

\section{Identifying the non-diagonal main terms}

We now consider, thinking of $z$ as fixed, the sum over $m$ and $n$ in \eqref{eq:9.10} 
above.   We use the three variable Mellin transform ${\widetilde {\mathcal W}_3^{\pm}}$
discussed in Lemma \ref{lemW3}.   Since we are summing 
over both choices of sign, the quantity in \eqref{eq:9.10} may be written as 
\begin{align} 
\label{eq:10.0}
\frac{Q^{1+\delta(\bsalpha,\bsbeta)}}{2} \frac{1}{(2\pi i)^3} \int_{(\epsilon)} 
\int_{(\frac 12+\epsilon)} \int_{(\frac 12+\epsilon)}   &
{\widetilde {\mathcal W}_3}(s_1,s_2;z)  
\frac{\zeta(1-z)}{\zeta(1+z)} Q^{\frac 32(s_1+s_2)-\frac z2} \nonumber\\
&\sum_{\substack{{m,n=1}}}^{\infty} 
\frac{\sigma(m;\bsalpha)\sigma(n;-\bsbeta)}{m^{\frac 12+s_1} n^{\frac 12+s_2}} 
\frac{g^z{\mathcal K}(-z;g,MN)}{\phi(gMN,1+z)} ds_2 ds_1 dz. 
 \end{align} 
Thus we are led to consider 
\begin{equation} 
\label{eq:10.1} 
\mathcal F(s_1,s_2;z) = 
\sum_{m,n =1}^{\infty} \frac{\sigma(m;\bsalpha)\sigma(n;-\bsbeta)}{m^{\frac 12+s_1} n^{\frac 12+s_2}} 
\frac{g^{z}}{\phi(gMN,1+z)} \mathcal K(-z;g, MN). 
\end{equation}  
The sum over $m$ and $n$ above has an obvious multiplicative structure, and 
we can therefore write 
$$ 
\mathcal F(s_1,s_2;z) = \prod_{p} {\mathcal F}_p(s_1,s_2;z), 
$$ 
where, recalling the definition of ${\mathcal K}$ from Lemma \ref{lemma9.1},  
\begin{align*}
\mathcal{F}_p(s_1,s_2;z) &= 
1+\frac{(p^z-1)}{p(p-1)}+
\sum_{\substack{{a,b \ge 0} \\ {\max(a,b)\ge 1} }} 
\frac{\sigma(p^a;\bsalpha)\sigma(p^b;-\bsbeta)}{p^{a(\frac 12+s_1)} p^{b(\frac 12+s_2)}} 
p^{z\min(a,b)} \frac{1-1/p^{1-z}}{1-1/p^{1+z}}  
\\
&+\frac{1}{(p-1)} \frac{p^z-1}{1-1/p^{1+z}}  \sum_{k=1}^{\infty} \frac{\sigma(p^k;\bsalpha)\sigma(p^k;-\bsbeta)}{p^{k(1+s_1+s_2-z)}}.
\end{align*}

The behavior of ${\mathcal F}_p$ is dominated by the contributions from $(a,b)=(1,0)$, $(0,1)$, 
$(1,1)$ and $k=1$  terms above.  Thus we write 
\begin{align*}
{\mathcal F}(s_1,s_2;z) &= \zeta(2-z) \prod_{j=1}^{3} 
\frac{\zeta(\tfrac 12+s_1+\alpha_j)}{\zeta(\tfrac 12+s_1+ \alpha_j+1-z)} 
\prod_{\ell=1}^{3} 
\frac{\zeta(\tfrac 12+s_2-\beta_\ell)}{\zeta(\tfrac 12+s_2 -\beta_\ell +1-z)}  \\
&\times\prod_{j,\ell=1}^{3} \zeta(1+s_1+s_2-z+\alpha_j-\beta_\ell) {\mathcal G}(s_1,s_2;z).
\end{align*}
Here ${\mathcal G}(s_1,s_2;z) = \prod_p {\mathcal G}_p(s_1,s_2;z)$ is absolutely 
convergent in a wider region  of $s_1$, $s_2$ and $z$.   

For a fixed value of $z$ with Re$(z)=\epsilon$, ${\mathcal F}(s_1,s_2;z)$ has 
 nine poles at $s_1 =\tfrac 12- \alpha_j$ and $s_2 =\tfrac 12 +\beta_\ell$.  
Keeping $z$ fixed, we move the lines of integration in $s_1$ and $s_2$ 
to Re$(s_1) =2\epsilon$ and Re$(s_2)=2\epsilon$, and pick up the 
residues at these poles, and the integrals on the remaining lines is acceptably small.   
To move the lines of integration carefully, we first truncate the integrals in $s_1$ and $s_2$ at 
a height $T=Q^2$ and then move the line of integration over $s_2$ first and then 
the line over $s_1$.  We use the estimate \eqref{eq:7.45} and find that 
the accumulated error may be bounded by $O(Q^{\frac 74+\epsilon})$.

Now we work out the contribution of the residues.  As noted above, there 
are nine such terms and for simplicity we shall work out the contribution 
from $s_1=\tfrac 12-\alpha_1$ and $s_2=\tfrac 12+\beta_1$, the other 
cases being similar.   Our goal is to show that this term contributes 
\begin{equation} 
\label{eq:10.11} 
H(0;\bsalpha,\bsbeta) \sum_{q} \Psi\Big(\frac qQ\Big) \phi^{\flat}(q) \int_{-\infty}^{\infty} {\mathcal Q}(q;\pi(\bsalpha)+iy, 
\pi(\bsbeta)+iy) dy + O(Q^{}), 
\end{equation} 
where $\pi \in S_6/(S_3\times S_3)$ is the permutation with $\pi(\bsalpha) = (\beta_1,\alpha_2,\alpha_3)$ and $\pi(\bsbeta) = (\alpha_1,\beta_2,\beta_3)$.   

The residue of ${\mathcal F}(s_1,s_2;z)$  equals 
\begin{equation} 
\label{eq:10.2}
\prod_{j=2}^{3} \frac{\zeta(1+\alpha_j-\alpha_1) \zeta(1+\beta_1-\beta_j)}{\zeta(2+\alpha_j-\alpha_1-z) 
\zeta(2-\beta_j+\beta_1-z)} \prod_{\substack{{j,\ell=1} \\ {(j,\ell)\neq (1,1)}}}^{3} 
\zeta(2+\beta_1-\alpha_1 -z+\alpha_j-\beta_\ell) {\mathcal G}(\tfrac 12-\alpha_1,\tfrac12+\beta_1;z). 
\end{equation}
We use this in \eqref{eq:10.0}, and move the line of integration in $z$ to Re$(z)= \frac 32-\epsilon$.  
In doing so, we encounter a simple pole at $z= 1-\alpha_1+\beta_1$ (from the 
${\widetilde {\mathcal W}}_3(\tfrac 12-\alpha_1,\tfrac 12+\beta_1;z)$ term) and the residue here 
is the dominant contribution.  Note that there are potential poles at $z= 1-\alpha_1+\beta_1+\alpha_j-\beta_\ell$ (for $(j,\ell)\neq (1,1)$) but these are offset by the corresponding zeros of 
$H((1-\alpha_1+\beta_1-z)/2;\bsalpha,\bsbeta)$ at these points.  
Taking the residue at $z=1-\alpha_1+\beta_1$, the expression \eqref{eq:10.2} 
simplifies to 
\begin{equation} 
\label{eq:10.3} 
\frac{{\mathcal Z}(\tfrac 12;\pi(\bsalpha),\pi(\bsbeta))}{\zeta(1+\beta_1 -\alpha_1)} {\mathcal G}(\tfrac 12-\alpha_1,\tfrac 12+\beta_1;1+\beta_1-\alpha_1). 
\end{equation}
Moreover, the residue of 
$$ 
\frac{Q^{1+\delta(\bsalpha,\bsbeta)}}{2} \widetilde{\mathcal W}_3(\tfrac 12-\alpha_1,\tfrac 12+\beta_1; z) 
\frac{\zeta(1-z)}{\zeta(1+z)} Q^{\frac32 (1+\beta_1-\alpha_1) -\frac z2} 
$$ 
at $z=1-\alpha_1+\beta_1$ gives, using \eqref{eq:7.44}, 
\begin{align*} 
\frac{Q^{2+\delta(\pi(\bsalpha),\pi(\bsbeta))}\zeta(\alpha_1-\beta_1)}{2 \pi^{\delta(\bsalpha,\bsbeta)}\zeta(2-\alpha_1+\beta_1)} 
&{\widetilde \Psi}(2+\delta(\pi(\bsalpha),\pi(\bsbeta)))  {H(0;\bsalpha,\bsbeta)}\nonumber \\ 
&\int_{-\infty}^{\infty} G(\tfrac 12;\bsalpha+it,\bsbeta+it) {\mathcal H}(\tfrac 12-\alpha_1-it,1-\alpha_1+\beta_1) dt.
\end{align*} 
Using now \eqref{eq:7.431} and the functional equation connecting $\zeta(\alpha_1-\beta_1)$ and 
$\zeta(1-\alpha_1+\beta_1)$, the above simplifies to give
\begin{align}
\label{eq:10.4}
\frac{Q^{2+\delta(\pi(\bsalpha),\pi(\bsbeta))}\zeta(1-\alpha_1+\beta_1)}{2 \pi^{\delta(\pi(\bsalpha),
\pi(\bsbeta))}\zeta(2-\alpha_1+\beta_1)} 
&{\widetilde \Psi}(2+\delta(\pi(\bsalpha),\pi(\bsbeta)))  {H(0;\bsalpha,\bsbeta)}\nonumber \\ 
&\times  \int_{-\infty}^{\infty}G(\tfrac 12;\pi(\bsalpha)+it,\pi(\bsbeta)+it)dt.
\end{align}
Combining \eqref{eq:10.3} and \eqref{eq:10.4} we obtain that the contribution of this term to \eqref{eq:10.3} is 
\begin{align} 
\label{eq:10.5} 
 \frac{Q^{2+\delta(\pi(\bsalpha),\pi(\bsbeta))}}{2\pi^{\delta(\pi(\bsalpha),\pi(\bsbeta))}} 
 &{\mathcal Z}(\tfrac 12;\pi(\bsalpha),\pi(\bsbeta)) \widetilde{\Psi}(2+\delta(\pi(\bsalpha),\pi(\bsbeta))) 
 H(0;\bsalpha,\bsbeta)\nonumber\\
&\times \frac{{\mathcal G}(\tfrac 12-\alpha_1,\tfrac 12+\beta_1;1+\beta_1-\alpha_1)}{\zeta(2-\alpha_1+\beta_1)}  \int_{-\infty}^{\infty} G(\tfrac 12;\pi(\bsalpha)+it,\pi(\bsbeta)+it) dt.
 \end{align}

It remains now to match up the quantity in \eqref{eq:10.5} above with our desired object in 
\eqref{eq:10.11}.  Let us begin by first simplifying the expression in 
\eqref{eq:10.11}.  The main term there equals 
$$ 
H(0;\bsalpha,\bsbeta) \sum_{q} \Psi\Big(\frac qQ\Big) \phi^{\flat}(q) 
\Big(\frac{q}{\pi}\Big)^{\delta(\pi(\bsalpha),\pi(\bsbeta))} 
\frac{{\mathcal {AZ}}}{{\mathcal B}_q} (\tfrac 12;\pi(\bsalpha),\pi(\bsbeta)) 
\int_{-\infty}^{\infty} G(\tfrac 12; \pi(\bsalpha)+it, \pi(\bsbeta)+it) dt.  
$$ 
Using that $\phi^{\flat}(q) = \tfrac 12 \phi^*(q) + O(1)$, and the 
function $\phi^*$ is multiplicative with $\phi^*(p)=p-2$ and $\phi^{*} (p^k) = p^{k-2}(p-1)^2$ 
for $k\ge 2$, we may evaluate using a standard contour shift argument the sum over 
$q$ above.  Thus the above becomes 
\begin{align}
\label{eq:10.6} 
H(0;\bsalpha,\bsbeta)&\Big( \int_{-\infty}^{\infty} G(\tfrac 12; \pi(\bsalpha)+it, \pi(\bsbeta)+it) dt \Big)
\frac{Q^{2+\delta(\pi(\bsalpha),\pi(\bsbeta))}}{2\pi^{\delta(\pi(\bsalpha),\pi(\bsbeta))}} 
{\widetilde \Psi}(2+\delta(\pi(\bsalpha),\pi(\bsbeta)))\nonumber \\
&\times {\mathcal {AZ}}(\tfrac 12;\pi(\bsalpha),\pi(\bsbeta)) 
\prod_p \Big( 1-\frac 1p \Big) \Big( 1+ \frac 1{{\mathcal B}_p(\tfrac 12;\pi(\bsalpha),\pi(\bsbeta))} \Big(\frac 1p -\frac 1{p^2}-\frac 1{p^3}
\Big)\Big). 
\end{align}

Comparing \eqref{eq:10.5} and \eqref{eq:10.6} we note that many of the terms 
match up, and what remains is to match up the Euler products on both sides.  
For this it suffices to check that the Euler factors at each prime match up.   This 
entails checking whether 
$$
\frac{{\mathcal G}_p(\frac 12-\alpha_1,\frac 12 +\beta_1;1+\beta_1-\alpha_1)}{\zeta_p(2-\alpha_1+\beta_1)} 
$$
equals 
$$  
 {\mathcal A}_p(\tfrac 12;\pi(\bsalpha),\pi(\bsbeta)) \Big(1-\frac 1p\Big) \Big( 1+ \frac 1{{\mathcal B}_p(\tfrac 12;\pi(\bsalpha),\pi(\bsbeta))} \Big(\frac 1p -\frac 1{p^2}-\frac 1{p^3}
\Big)\Big)?
$$ 
 A little calculation reduces this to checking whether 
 $$ 
 {\mathcal F}_p(\tfrac 12-\alpha_1,\tfrac 12+\beta_1;1+\beta_1-\alpha_1) \Big(1-\frac 1p\Big) 
 \frac{\zeta_p(1-\alpha_1+\beta_1)}{\zeta_p(2-\alpha_1+\beta_1)} 
 $$ 
 equals 
 $$
 {\mathcal B}_p (\tfrac 12;\pi(\bsalpha),\pi(\bsbeta)) + \frac 1p - \frac 1{p^2} -\frac {1}{p^3}? 
 $$ 
 With a little 
calculation, this may be verified upon using the following Parseval identities 
 \begin{align*}
&  \sum_{a, b\ge 0} \frac{\sigma(p^a;\bsalpha) \sigma(p^b;-\bsbeta)}{p^{a(1-\alpha_1)} 
  p^{b(1+\beta_1)} } p^{(1-\alpha_1+\beta_1) \min(a,b)} \\
&=\int_0^1 \Big(\sum_{a=0}^{\infty} \frac{\sigma(p^a;\bsalpha)e(a\theta)}{p^{a(1-\alpha_1-\beta_1)/2}} 
\Big) \Big(\sum_{b=0}^{\infty} \frac{\sigma(p^b;-\bsbeta)e(-b\theta)}{p^{b(1+\alpha_1+\beta_1)/2}}\Big) 
\Big( 1+\sum_{k=1}^{\infty} \frac{e(k\theta)}{p^{k(1-\alpha_1+\beta_1)/2}} + 
\sum_{\ell=1}^{\infty} \frac{e(-\ell \theta)}{p^{\ell(1-\alpha_1+\beta_1)/2}} \Big)d\theta,\\
%&= \int_0^1 
%\prod_{j=1}^{3} \Big( 1-\frac{e(\theta)}{p^{}}\Big)^{-1} \Big(1- \frac{e(-\theta)}{p^{}} \Big)^{-1} 
%\Big( \Big) d\theta,
\end{align*}
and 
$$ 
\sum_{k=0}^{\infty} \frac{\sigma(p^k;\bsalpha)\sigma(p^k;-\bsbeta)}{p^k}  
= \int_0^1 \prod_{j=1}^{3} \Big( 1- \frac{e(\theta)}{p^{\frac 12+\alpha_j}}\Big)^{-1} 
\Big(1- \frac{e(-\theta)}{p^{\frac 12-\beta_j}}\Big)^{-1}d\theta.
$$ 

\section{Completion of the proof of Theorem 1} 

\noindent Recall from \S 4 that our goal is to evaluate 
${\mathcal I}(\Psi,Q;\bsalpha,\bsbeta) = \Delta(\Psi,Q;\bsalpha,\bsbeta) + 
\Delta(\Psi,Q;\bsalpha,\bsbeta)$ (see \eqref{eqn:4}, \eqref{Delta} and 
\eqref{eqn:5}).   We then decomposed $\Delta(\Psi,Q;\bsalpha,\bsbeta) 
= {\mathcal D}(\Psi,Q;\bsalpha,\bsbeta) + {\mathcal S}(\Psi,Q;\bsalpha,\bsbeta) 
+ {\mathcal G}(\Psi,Q;\bsalpha,\bsbeta)$ (see \eqref{eqn:mainsum}).  
In Lemma \ref{diagonal} we evaluated ${\mathcal D}(\Psi,Q;\bsalpha,\bsbeta)$ obtaining 
one of the twenty terms in our desired asymptotic formula, together with an error term of $O(Q^{\frac 54+\epsilon})$.  

The work in \S 5 extracts a main term ${\mathcal {MS}}(\Psi,Q;\bsalpha,\bsbeta)$ 
out of ${\mathcal S}(\Psi, Q;\bsalpha,\bsbeta)$ with an error term of $O(Q^{2+\epsilon}/D)$.  
Correspondingly in \S 6 we extract a main term $\mathcal {MG}(\Psi,Q;\bsalpha,\bsbeta)$ 
out of ${\mathcal G}(\Psi,Q;\bsalpha,\bsbeta)$ with an error that is estimated in 
\S 8 as $O(Q^{\frac 74+\epsilon} D^{\frac 32})$.   The two main terms 
${\mathcal {MS}}$ and ${\mathcal {MG}}$ are combined in \S 9 to obtain a 
single term given in \eqref{eq:9.10} with an error $O(Q^{2+\epsilon}D + Q^{\frac 32+\epsilon}D)$.  
The choice $D= Q^{\frac 1{10}}$ minimizes our errors, and the total error is $O(Q^{\frac {19}{10}+\epsilon})$.

The main term of \eqref{eq:9.10} is evaluated in \S 10 by an involved residue 
calculation.  Nine terms in our desired asymptotic formula arise here, 
corresponding to the nine transpositions $\pi = (\alpha_j\beta_\ell)$.  
Thus $\Delta(\Psi,Q;\bsalpha,\bsbeta)$ leads to ten terms in our 
desired asymptotic formula, and the remaining ten come from the 
$\Delta(\Psi,Q;\bsbeta,\bsalpha)$ contribution.

Putting everything together we conclude that 
$$ 
H(0;\bsalpha,\bsbeta) \sum_q \sumflat_{\chi\pmod q} 
\Psi\Big(\frac q{Q}\Big) 
\int_{-\infty}^{\infty} \Lam(\chi;\bsalpha+it, \bsbeta+it) dt 
$$ 
equals 
$$ 
H(0;\bsalpha,\bsbeta) \sum_{q} \Psi\Big(\frac qQ\Big) \phi^{\flat}(q) 
\int_{-\infty}^{\infty} {\widetilde {\mathcal Q}} (q;\bsalpha+it,\bsbeta+it) dt + O(Q^{\frac{19}{10} +
\epsilon}). 
$$ 
If the $\alpha_j$ and $\beta_\ell$ are bounded away from each other, and 
from zero in such a way that $H(0;\bsalpha,\bsbeta) \gg Q^{-\epsilon}$ then 
we may divide both sides by $H(0;\bsalpha,\bsbeta)$ and obtain the 
Theorem in this case.   The general case follows from this by noting that 
the two expressions above must both be analytic in the variables $\alpha_j$ and 
$\beta_\ell$.

%\section{Conclusion}
%It seems that it may be possible to treat the eighth moment of
%Dirichlet $L$-functions, corresponding to the second of Huxley's
%formulae from the introduction. In a future paper we may return to
%the question of deriving the corresponding asymptotic formula for
%the eighth moment. However, it is clear from this paper that we
%cannot obtain the full eighth moment conjecture with a power savings
%from the techniques presented here. In particular, there will arise
%new main terms, in such a calculation, from terms that have been
%relegated to error terms here.

 \end{document}